\documentclass{amsart}

\usepackage{thesis}  % personal style file

\newcommand{\cart}{\textup{cart}}
\newcommand{\cocart}{\textup{cocart}}
\newcommand{\dec}{\star}
\newcommand{\Funp}{\Fun_{**}}
\newcommand{\Funpdec}{\Fun_{**}^\dec}
\newcommand{\uq}{\textup{u.q.}}

\begin{document}

\title{Quillen adjunctions induce adjunctions of quasicategories}
\date{\today}
\author{\textsc{Aaron Mazel-Gee}}
\maketitle

\begin{abstract}
We prove that a Quillen adjunction of model categories (of which we do not require functorial factorizations and of which we only require finite bicompleteness) induces a canonical adjunction of underlying quasicategories.
\end{abstract}

\setcounter{tocdepth}{1}
\tableofcontents

\setcounter{section}{-1}

\section{Introduction}

\subsection*{Background and motivation}

Broadly speaking, the methods of abstract homotopy theory can be divided into two types: those that work \textit{internally} to a given homotopy theory, and those that work \textit{externally} with all homotopy theories at once.  By far the most prominent method of the first type is the theory of model categories, introduced by Quillen in his seminal work \cite{QuillenHA}.  On the other hand, there are now a plethora of models for ``the homotopy theory of homotopy theories'', all of them equivalent in an essentially unique manner (reviewed briefly in \cref{section conventions}); for the moment, we will refer to such objects collectively as ``$\infty$-categories''.

However, there is some apparent overlap between these two situations: model categories do not exist in isolation, but can be related by Quillen adjunctions and Quillen equivalences.  We are thus led to a natural question.

\begin{qn}\label{main question}
Precisely what $\infty$-categorical phenomena do Quillen adjunctions and Quillen equivalences encode?
\end{qn}

Of course, one expects that Quillen adjunctions induce ``adjunctions of $\infty$-categories'' and that Quillen equivalences induce ``equivalences of $\infty$-categories''.  However, it turns out that actually making these statements precise i a subtle task.  On the other hand, it is made easier by imposing various additional assumptions or by settling for more modest conclusions, and hence there already exist an assortment of partial results in this direction.  We defer a full history to \cref{section history}; the state of affairs can be summarized as follows.
\begin{itemizesmall}
\item Quillen equivalences are known to induce weak equivalences of $s\Set$-enriched categories (where $s\Set$ denotes the category of simplicial sets, and by ``weak equivalence'' we mean in the Bergner model structure).
\item Quillen adjunctions are known to induce adjunctions of homotopy categories, and are more-or-less known to induce adjunctions of $\ho(s\Set_\KQ)$-enriched homotopy categories (where $s\Set_\KQ$ denotes the category of simplicial sets equipped with the standard Kan--Quillen model structure).
\item Quillen adjunctions between model categories that admit suitable co/fibrant replacement functors are more-or-less known to induce adjunctions of quasicategories.
\item Simplicial Quillen adjunctions between simplicial model categories are known to induce adjunctions of quasicategories, and moreover certain Quillen adjunctions can be replaced by simplicial Quillen adjunctions of simplicial model categories.
\end{itemizesmall}

Thus, in order to fully unify the internal and external approaches to abstract homotopy theory, it remains to show that an arbitrary Quillen adjunction induces an adjunction of $\infty$-categories.  The purpose of the present paper is to prove this assertion when we take the term ``$\infty$-category'' to mean ``quasicategory''. % This result is surely a bit of folklore, but it does not seem to appear in full generality anywhere in the literature.

Since model categories have figured so foundationally into much of the development of axiomatic homotopy theory, it seems that Quillen adjunctions are generally viewed as such basic and fundamental objects that they hardly merit further interrogation.  However, inasmuch as there is a far deeper understanding today of ``the homotopy theory of homotopy theories'' than existed in 1967, we consider it to be a worthwhile endeavor to settle this matter once and for all.

\begin{rem}
An adjunction of $s\Set$-enriched categories induces an adjunction of quasicategories (see \cite[Corollary 5.2.4.5]{LurieHTT}), but the converse is presumably false: an adjunction of $s\Set$-enriched categories is by its very nature extremely rigid -- making reference to simplicial sets up to isomorphism, with no mention of their ambient model structure --, whereas an adjunction of quasicategories is a much more flexible notion and incorporates a wealth of higher coherence data.  (Of course, both of these notions are strictly stronger than that of an adjunction of $\ho(s\Set_\KQ)$-enriched categories.)

In fact, the datum of ``an adjunction of quasicategories'' only specifies the actual adjoint functors themselves up to contractible spaces of choices.\footnote{We refer the reader to \cite[\sec 5.2]{LurieHTT} for a thorough exposition of the theory of adjunctions of quasicategories.}  This situation may appear somewhat abstruse to those not familiar with the theory of quasicategories, but in our view, quasicategories were never really meant to be worked with at the simplex-by-simplex level anyways: they function best when manipulated via (quasicategorical) universal properties, the praxis of which is actually quite similar to that of 1-categories in many ways.  So, all in all, we view this situation primarily as a reaffirmation of the philosophy of quasicategories: that it's too much to demand strict composition in the first place, and that working with rigid models can obscure the essential features of the true and underlying mathematics.
\end{rem}

\subsection*{Acknowledgments}

We cordially thank Zhen Lin Low for a lively and extended discussion regarding the material presented in this paper, as well as Dave Carchedi, Bill Dwyer, Geoffroy Horel, Tyler Lawson, Thomas Nikolaus, and Emily Riehl for their helpful input.  We also gratefully acknowledge the financial support provided by UC Berkeley's geometry and topology RTG (grant DMS-0838703) during the time that this work was carried out.

\section{Notation and conventions}\label{section conventions}

\subsection{Specific categories}

As we have already indicated, we write $s\Set$ for the category of simplicial sets.  Of course, this is because we write $\Set$ for the category of sets and we write $c(-)$ and $s(-)$ for categories of co/simplicial objects; hence, we will write $ss\Set$ for the category of bisimplicial sets.  We also write $\Cat$ for the category of categories, $\RelCat$ for the category of relative categories, and $\Cat_{s\Set}$ for the category of $s\Set$-enriched categories.  We will write $\Nerve : \Cat \ra s\Set$ for the usual nerve functor.

We will consider categories as special instances of both relative categories and $s\Set$-enriched categories: for the former we consider $\Cat \subset \RelCat$ by endowing our categories with the \textit{minimal} relative structure (in which only the identity maps are marked as weak equivalences), and for the latter we consider $\Cat \subset \Cat_{s\Set}$ via the inclusion $\Set \subset s\Set$ of sets as discrete simplicial sets.

\subsection{Specific model categories}

As we have already indicated, we will model ``spaces'' using the standard \bit{Kan--Quillen model structure} $s\Set_\KQ$, while to model ``the homotopy theory of homotopy theories'', we will make use of all four of
\begin{itemizesmall}
\item the \bit{Rezk model structure} (a/k/a the ``complete Segal space'' model structure) $ss\Set_\Rezk$ of \cite[Theorem 7.2]{RezkCSS},
\item the \bit{Barwick--Kan model structure} $\RelCat_\BarKan$ of \cite[Theorem 6.1]{BK-relcats},
\item the \bit{Bergner model structure} $(\Cat_{s\Set})_\Bergner$ of \cite[Theorem 1.1]{Bergner}, and
\item the \bit{Joyal model structure} $s\Set_\Joyal$ of \cite[Theorem 2.2.5.1]{LurieHTT}.
\end{itemizesmall}
As explained in \cite{BSP}, these are all equivalent in an essentially unique way, though the meanings of the phrases ``equivalent'' and ``essentially unique'' here are both slightly subtle.

We will use the following equivalences between these models for ``the homotopy theory of homotopy theories''.
\begin{itemize}

\item The Barwick--Kan model structure is defined by lifting the cofibrantly generated model structure $ss\Set_\Rezk$ along an adjunction $ss\Set \adjarr \RelCat$ (so that the right adjoint creates the fibrations and weak equivalences), which then becomes a Quillen equivalence (see \cite[Theorem 6.1]{BK-relcats}).  Then, the \bit{Rezk nerve} functor $\NerveRezk : \RelCat \ra ss\Set$ of \cite[3.3]{RezkCSS} (there called the ``classification diagram'' functor) admits a natural weak equivalence in $s(s\Set_\KQ)_\Reedy$ to this right Quillen equivalence (see \cite[Lemma 5.4]{BK-relcats}).  Thus, in light of the left Bousfield localization $s(s\Set_\KQ)_\Reedy \adjarr ss\Set_\Rezk$, we see that the Rezk nerve defines a relative functor $\NerveRezk : \RelCat_\BarKan \ra ss\Set_\Rezk$ which creates the weak equivalences in $\RelCat_\BarKan$.

For any relative category $(\R,\bW_\R)$, we will write $\Fun([n],\R)^{\bW_\R}$ for the wide subcategory on the levelwise weak equivalences in $\Fun([n],\R)$, the nerve of which is precisely $\NerveRezk(\R,\bW_\R)_n$.

\item The \bit{hammock localization} functor $\ham : \RelCat \ra \Cat_{s\Set}$ of \cite[2.1]{DKCalc} defines a weak equivalence in $\RelCat_\BarKan$ on the underlying relative categories of the model categories $\RelCat_\BarKan$ and $(\Cat_{s\Set})_\Bergner$ (see \cite[Theorem 1.7]{BK-simploc}).

\item The \bit{homotopy-coherent nerve} functor $\Nervehc : \Cat_{s\Set} \ra s\Set$ of \cite[Definition 1.1.5.5]{LurieHTT} (there called the ``simplicial nerve'' functor, originally defined in \cite{Cordier}, there called the ``nerf homotopiquement coh\'{e}rent'' functor) defines a right Quillen equivalence $(\Cat_{s\Set})_\Bergner \ra s\Set_\Joyal$ (see \cite[Theorem 2.2.5.1]{LurieHTT}).

\end{itemize}

Since the model category $(\Cat_{s\Set})_\Bergner$ is cofibrantly generated, it comes naturally equipped with a fibrant replacement functor.  However, it will be convenient for us to use one which does not change the objects.  Thus, for definiteness we define $\bbR_\Bergner : (\Cat_{s\Set})_\Bergner \ra (\Cat_{s\Set})_\Bergner$ to be the functor given by applying Kan's $\Ex^\infty$ functor locally, i.e.\! to each hom-object.  (Note that $\Ex^\infty$ preserves finite products, being a filtered colimit of right adjoints.)  We now define the \bit{underlying quasicategory} functor to be the composite
\[ \uq : \RelCat \xra{\ham} \Cat_{s\Set} \xra{\bbR_\Bergner} \Cat_{s\Set} \xra{\Nervehc} s\Set . \]
As $\Nervehc$ is a right Quillen functor, this does indeed take values in quasicategories (and defines a relative functor $\RelCat_\BarKan \ra s\Set_\Joyal$).

\subsection{General model categories}

A model category $\C$ comes equipped with various attendant subcategories, for which we must fix some notation.  We will write
\begin{itemizesmall}
\item $\bW_\C \subset \C$ for the subcategory of weak equivalences,
\item $\C^c , \C^f,\C^{cf} \subset \C$ for the full subcategories of cofibrant, fibrant, and bifibrant objects, respectively,
\item $\bW^c_\C = \C^c \cap \bW_\C \subset \C$ and $\bW^f_\C = \C^f \cap \bW_\C \subset \C$,
\end{itemizesmall}
and similarly for other model categories.  We will use the arrows $\cofibn$ and $\fibn$ to denote cofibrations and fibrations, respectively, and we will decorate an arrow with the symbol $\approx$ to denote that it is a weak equivalence.

\subsection{Foundations}

Throughout, we will ignore all set-theoretic issues.  These are irrelevant to our aims, and in any case they can be dispensed with by appealing to the usual device of \textit{Grothendieck universes} (see \cite[\sec 1.2.15]{LurieHTT}).

\section{The main theorem}\label{section main theorem}

Let $F : \C \adjarr \D : G$ be a Quillen adjunction of model categories.  Note that the functors $F$ and $G$ do not generally define functors of underlying relative categories: they do not generally preserve weak equivalences.  Nevertheless, all is not lost: the inclusions $(\C^c,\bW^c_\C) \hookra (\C,\bW_\C)$ and $(\D^f,\bW^f_\D) \hookra (\D,\bW_\D)$ are weak equivalences in $\RelCat_\BarKan$ (as is proved in \cref{inclusion of co/fibrants induces BK weak equivalence}), and moreover by Kenny Brown's lemma (or rather its immediate consequence \cite[Corollary 7.7.2]{Hirsch}), the composites
\[ F^c : \C^c \hookra \C \xra{F} \D \]
and
\[ \C \xla{G} \D \hookla \D^f : G^f \]
do preserve weak equivalences.  Of course, this presents a problem: these two functors no longer have opposite sources and targets!  Despite this, we have the following theorem, which is the main result of this paper.

\begin{thm}\label{main theorem}
The functors $F^c$ and $G^f$ induce a canonical adjunction between the underlying quasicategories $\uq(\C)$ and $\uq(\D)$, informally denoted by $\uq(F^c) : \uq(\C) \adjarr \uq(\D) : \uq(G^f)$.
\end{thm}

Recall that an adjunction of quasicategories is precisely a map $\M \to \Delta^1$ which is both a cocartesian fibration and a cartesian fibration: the left adjoint is then its unstraightening $\M_0 \ra \M_1$ as a \textit{cocartesian} fibration, while the right adjoint is its unstraightening $\M_0 \la \M_1$ as a \textit{cartesian} fibration.  Thus, the first step in proving \cref{main theorem} is to obtain a cocartesian fibration over $\Delta^1$ which models $F^c$ and a cartesian fibration over $\Delta^1$ which models $G^f$.  We will actually define these as $s\Set$-enriched categories over $[1]$, relying on a recognition result (\cite[Proposition 5.2.4.4]{LurieHTT}) to deduce that these induce co/cartesian fibrations of quasicategories over $\Delta^1$.

\begin{constr}\label{construct co/cartesian fibrations}
We define the object $\cocart(\ham(F^c)) \in (\Cat_{s\Set})_{/[1]}$ as follows:
\begin{itemize}
\item the fiber over $0 \in [1]$ is $\ham(\C^c)$;
\item the fiber over $1 \in [1]$ is $\ham(\D)$;
\item for any $x \in \ham(\C^c)$ and any $y \in \ham(\D)$, we set $\hom_{\cocart(\ham(F^c))}(y,x) = \es$ and
\[ \hom_{\cocart(\ham(F^c))}(x,y) = \hom_{\ham(\D)}(F^c(x),y) . \]

\end{itemize}
Composition within the fibers is immediate, and is otherwise given by composition in $\ham(\D)$.

Similarly, we define the object $\cart(\ham(G^f)) \in (\Cat_{s\Set})_{/[1]}$ as follows:
\begin{itemize}
\item the fiber over $0 \in [1]$ is $\ham(\C)$;
\item the fiber over $1 \in [1]$ is $\ham(\D^f)$;
\item for any $x \in \ham(\C)$ and any $y \in \ham(\D^f)$, we set $\hom_{\cart(\ham(G^f))}(y,x) = \es$ and
\[ \hom_{\cart(\ham(G^f))}(x,y) = \hom_{\ham(\C)}(x,G^f(y)) . \]
\end{itemize}
Again composition within the fibers is immediate, but this time it is otherwise given by composition in $\ham(\C)$.
\end{constr}

Now, it is actually not so hard to show (using \cite[Proposition 5.2.4.4]{LurieHTT}) that $\cocart(\ham(F^c))$ and $\cart(\ham(G^f))$ each give rise to an adjunction of quasicategories.\footnote{Used in such a way, \cite[Proposition 5.2.4.4]{LurieHTT} becomes a quasicategorical analog of the dual of \cite[Chapter IV, \sec 1, Theorem 2(ii)]{MacLaneCWM}: given a functor $F_1 : \C_1 \ra \C_2$ of small 1-categories, a right adjoint is freely generated by choices, for all $c_2 \in \C_2$, of objects $F_2(c_2) \in \C_1$ and morphisms $F_1(F_2(c_2)) \ra c_2$ in $\C_2$ inducing natural isomorphisms $\hom_{\C_1}(-,F_2(c_2)) \xra{\cong} \hom_{\C_2}(F_1(-),c_2)$.  (Of course, both the category of right adjoints to $F_1$ and the category of such data are $(-1)$-connected groupoids in any case; we are making the evil assertion that they are actually \textit{equal}.)}  However, a priori, such an argument might give two possibly different adjunctions!  In order to show that they actually agree, we introduce the following intermediate object.

\begin{constr}\label{construct collage}
We define the object $(\C^c + \D^f,\bW_{\C^c + \D^f}) \in \RelCat_{/[1]}$ as follows:
\begin{itemize}

\item the fiber over $0 \in [1]$ is $(\C^c,\bW_\C^c)$;

\item the fiber over $1 \in [1]$ is $(\D^f,\bW_\D^f)$;

\item for any $x \in \C^c$ and any $y \in \D^f$, we set $\hom_{\C^c + \D^f}(y,x) = \es$ and
\[ \hom_{\C^c + \D^f}(x,y) = \hom_\C(x,G (y)) \cong \hom_\D(F (x) , y) , \]
declaring none of these maps to be weak equivalences.

\end{itemize}
Composition within fibers is immediate, and is otherwise given by composition in either $\C^c$ or $\D^f$, whichever contains two of the three objects involved.  We will depict arrows living over the unique non-identity map in $[1]$ by
\[ \begin{tikzcd}
x \arrow[rightsquigarrow]{r} & y,
\end{tikzcd} \]
and we will refer to such arrows as \bit{bridge arrows}.
\end{constr}

Applying the hammock localization functor to \cref{construct collage} gives rise to an object
\[ \ham(\C^c+\D^f) \in (\Cat_{s\Set})_{/[1]} , \]
and this will be what connects the two objects of \cref{construct co/cartesian fibrations}.  In order to see how this works, we must examine the $s\Set$-enriched category $\ham(\C^c+\D^f)$.  First of all, its fiber over $0 \in [1]$ is precisely $\ham(\C^c)$, while its fiber over $1 \in [1]$ is precisely $\ham(\D^f)$.  On the other hand, for $x \in \C^c$ and $y \in \D^f$, the simplicial set $\hom_{\ham(\C^c+\D^f)}(x,y)$ has as its $n$-simplices the reduced hammocks of width $n$ in the relative category $(\C^c+\D^f,\bW_{\C^c+\D^f})$; since none of the bridge arrows are weak equivalences, such a hammock must be of the form
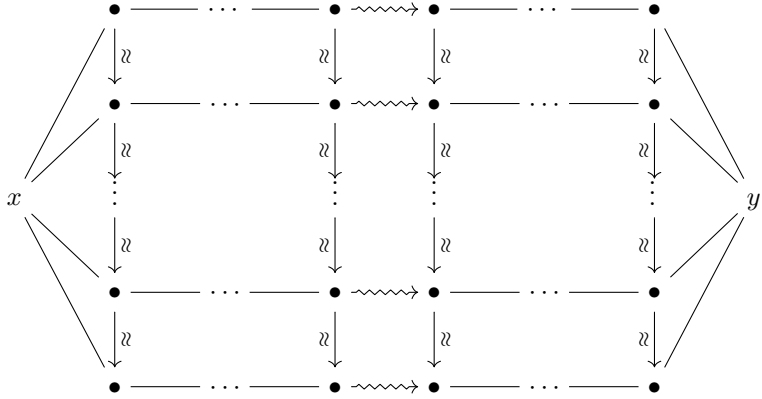
\[ \begin{tikzcd}
& \bullet \arrow[-]{r} \arrow{d}[sloped, anchor=south]{\approx} & \cdots \arrow[-]{r} & \bullet \arrow[rightsquigarrow]{r} \arrow{d}[sloped, anchor=north]{\approx} & \bullet \arrow[-]{r} \arrow{d}[sloped, anchor=south]{\approx} & \cdots \arrow[-]{r} & \bullet \arrow{d}[sloped, anchor=north]{\approx} \\
& \bullet \arrow[-]{r} \arrow{d}[sloped, anchor=south]{\approx} & \cdots \arrow[-]{r} & \bullet \arrow[rightsquigarrow]{r} \arrow{d}[sloped, anchor=north]{\approx} & \bullet \arrow[-]{r} \arrow{d}[sloped, anchor=south]{\approx} & \cdots \arrow[-]{r} & \bullet \arrow{d}[sloped, anchor=north]{\approx} \\
x \arrow[-]{ruu} \arrow[-]{ru} \arrow[-]{rd} \arrow[-]{rdd} & \vdots \arrow{d}[sloped, anchor=south]{\approx} & & \vdots \arrow{d}[sloped, anchor=north]{\approx} & \vdots \arrow{d}[sloped, anchor=south]{\approx} & & \vdots \arrow{d}[sloped, anchor=north]{\approx} & y \arrow[-]{luu} \arrow[-]{lu} \arrow[-]{ld} \arrow[-]{ldd} \\
& \bullet \arrow[-]{r} \arrow{d}[sloped, anchor=south]{\approx} & \cdots \arrow[-]{r} & \bullet \arrow[rightsquigarrow]{r} \arrow{d}[sloped, anchor=north]{\approx} & \bullet \arrow[-]{r}  \arrow{d}[sloped, anchor=south]{\approx} & \cdots \arrow[-]{r} & \bullet \arrow{d}[sloped, anchor=north]{\approx} \\
& \bullet \arrow[-]{r} & \cdots \arrow[-]{r} & \bullet \arrow[rightsquigarrow]{r} & \bullet \arrow[-]{r} & \cdots \arrow[-]{r} & \bullet \\
\end{tikzcd} \]
(where everything to the left of the column of bridge arrows lies in $\C^c$, while everything to the right of the column of bridge arrows lies in $\D^f$).  Hence, there are two maps
\[ \hom_{\cocart(\ham(F^c))}(x,y) \la \hom_{\ham(\C^c+\D^f)}(x,y) \ra \hom_{\cart(\ham(G^f))}(x,y) , \]
in which the left-pointing arrow is obtained by applying the relative functor $\C^c \xra{F^c} \D$ to the ``left half'' of the above hammock, while the right-pointing arrow is obtained by applying the relative functor $\C \xla{G^f} \D^f$ to the ``right half'' of the above hammock.  In fact, it is not hard to see that this respects composition of hammocks, and hence we obtain a diagram
\[ \cocart(\ham(F^c)) \la \ham(\C^c+\D^f) \ra \cart(\ham(G^f)) \]
in $(\Cat_{s\Set})_{/[1]}$.

The main ingredient in the proof of \cref{main theorem} is the following result.

\begin{prop}\label{main ingredient}
The horizontal maps
\[ \begin{tikzcd}
\cocart(\ham(F^c)) \arrow{rd} & \ham(\C^c+\D^f) \arrow{l}[swap]{\approx} \arrow{r}{\approx} \arrow{d} & \cart(\ham(G^f)) \arrow{ld} \\
& {[1]}
\end{tikzcd} \]
are weak equivalences in $(\Cat_{s\Set})_\Bergner$.
\end{prop}

We defer the proof of \cref{main ingredient} to \cref{section proof of main ingredient}.  In essence, we will show that the relative category $(\C^c+\D^f)$ admits a ``homotopical three-arrow calculus'', and then we will use this to show that the hom-objects in $\ham(\C^c+\D^f)$ can be computed using the co/simplicial resolutions of \cite{DKFunc}.  For an object $x \in \C^c$ equipped with a cosimplicial resolution $x^\bullet \in c(\C^c)$ and an object $y \in \D^f$ equipped with a simplicial resolution $y_\bullet \in s(\D^f)$, the isomorphisms
\[ \hom^\lw_\D(F^c(x^\bullet),y_\bullet) \cong \hom^\lw_{(\C^c+\D^f)}(x^\bullet,y_\bullet) \cong \hom^\lw_\C(x^\bullet,G^f(y_\bullet)) \]
of bisimplicial sets (where the superscript ``$\lw$'' stands for ``levelwise'') will, in light of the above observations and upon passing to diagonals, yield weak equivalences
\[ \hom_{\ham(\D)}(F(x),y) = \hom_{\cocart(\ham(F^c))}(x,y) \xla{\approx} \hom_{\ham(\C^c+\D^f)}(x,y) \xra{\approx} \hom_{\cart(\ham(G^f))}(x,y) = \hom_{\ham(\C)}(x,G(y)) \]
in $s\Set_\KQ$ (which are appropriately compatible with the given maps of hom-objects).

Using \cref{main ingredient}, we now prove the main theorem.

\begin{proof}[Proof of \cref{main theorem}]
First of all, using the recognition result \cite[Proposition 5.2.4.4]{LurieHTT} for when a fibrant object in $((\Cat_{s\Set})_\Bergner)_{/[1]}$ induces a cartesian fibration of quasicategories and its dual, it is immediate that the map $\bbR_\Bergner(\cocart(\ham(F^c))) \ra [1]$ induces a cocartesian fibration corresponding to $\uq(F^c)$ and that the map $\bbR_\Bergner(\cart(\ham(G^f))) \ra [1]$ induces a cartesian fibration corresponding to $\uq(G^f)$.\footnote{Essentially, \cite[Proposition 5.2.4.4]{LurieHTT} asks that every object in the fiber over $1 \in [1]$ admit a coreflection in the fiber over $0 \in [1]$.}\footnote{Note that fibrancy in $((\Cat_{s\Set})_\Bergner)_{/[1]}$ is created in $(\Cat_{s\Set})_\Bergner$ (see \cite[Theorem A.3.2.24(2)]{LurieHTT}).}  Moreover, condition (2) of \cite[Proposition 5.2.4.4]{LurieHTT} is clearly invariant under weak equivalence between fibrant objects in $(\Cat_{s\Set})_\Bergner)_{/[1]}$.  Thus, it follows from \cref{main ingredient} that the map
\[ \bbR_\Bergner(\ham(\C^c + \D^f)) \ra [1] \]
in $\Cat_{s\Set}$ induces both a cartesian fibration and a cocartesian fibration of quasicategories: that is, it induces an adjunction
\[ \uq(\C^c+\D^f) \ra \Delta^1 \]
of quasicategories, whose left adjoint can be identified with $\uq(F^c)$ and whose right adjoint can be identified with $\uq(G^f)$.  Finally, the fact that the inclusions $(\C^c,\bW^c_\C) \hookra (\C,\bW_\C)$ and $(\D^f,\bW^f_\D ) \hookra (\D,\bW_\D)$ are weak equivalences in $\RelCat_\BarKan$ (and hence induce weak equivalences in $s\Set_\Joyal$ of underlying quasicategories) follows from \cref{inclusion of co/fibrants induces BK weak equivalence} below.  Hence, choosing retractions in $s\Set_\Joyal$ as indicated, we obtain an adjunction
\[ \begin{tikzcd}
\uq(\C) \arrow[hookleftarrow]{r}{\approx} \arrow[dashed, out=20, in=173]{r} & \uq(\C^c) \cong \uq(\C^c+\D^f)_0 \arrow[hookrightarrow]{r} & \uq(\C^c+\D^f) \arrow[hookleftarrow]{r} \arrow{d} & \uq(\C^c+\D^f)_1 \cong \uq(\D^f) \arrow[hookrightarrow]{r}{\approx} & \uq(\D) \arrow[dashed, out=160, in=7]{l} \\
& & {[1]}
\end{tikzcd} \]
-- which might denoted informally as $\uq(F^c) : \uq(\C) \adjarr \uq(\D) : \uq(G^f)$ --, precisely as claimed.
\end{proof}

\begin{rem}
In general, the property of being a co/cartesian fibration over $S \in s\Set$ is not invariant under weak equivalence between inner fibrations in $(s\Set_\Joyal)_{/S}$.  However, it becomes invariant in the special case that $S = \Delta^1$, a fact we've exploited in the proof of \cref{main theorem} (through our usage of \cite[Proposition 5.2.4.4]{LurieHTT}).  Roughly speaking, this follows from the paucity of nondegenerate edges in $\Delta^1$.  Indeed, recall that given an inner fibration $X \ra S$:
\begin{itemize}
\item an edge $\Delta^1 \ra X$ is \textit{cartesian} (with respect to $X \ra S$) if it satisfies some universal property defined in terms of all of $X$ and $S$ (see \cite[Definition 2.4.1.1 and Remark 2.4.1.9]{LurieHTT});
\item an edge of $\Delta^1 \ra X$ is \textit{locally cartesian} if the induced edge $\Delta^1 \ra \Delta^1 \times_S X$ is cartesian with respect to the inner fibration $\Delta^1 \times_S X \ra \Delta^1$ (see \cite[Definition 2.4.1.11]{LurieHTT});
\item the map $X \ra S$ is a (resp.\! \textit{locally}) \textit{cartesian fibration} if it has a sufficient supply of (resp.\! locally) cartesian edges (see \cite[Definitions 2.4.2.1 and 2.4.2.6]{LurieHTT});
\item a locally cartesian fibration is a cartesian fibration if and only if the locally cartesian edges are ``closed under composition'' in the strongest possible sense (see \cite[Proposition 2.4.2.8]{LurieHTT});
\item if an edge of $X$ maps to an equivalence in $S$, then that edge is cartesian if and only if it is also an equivalence (see \cite[Proposition 2.4.1.5]{LurieHTT});
\item in light of the universal property defining cartesian edges, pre- or post-composing a cartesian edge in $X$ with an equivalence which projects to a degenerate edge in $S$ clearly yields another cartesian edge.
\end{itemize}
(The notions of cartesian and locally cartesian fibrations are the quasicategorical analogs of the 1-categorical notions of ``Grothendieck fibrations'' and ``Grothendieck prefibrations''.)
\end{rem}

\begin{rem}
One might also wonder about the possibility of using the objects $\cocart(F^c) = (\C^c+\D)$ and $\cart(G^f) = (\C+\D^f)$ of $\RelCat_{/[1]}$ in order to prove \cref{main theorem}.  In fact, it is not so hard to show that the inclusions $\cocart(F^c) \hookla (\C^c+\D^f) \hookra \cart(G^f)$ are weak equivalences in $\RelCat_\BarKan$, and moreover there are natural maps $\ham(\cocart(F^c)) \ra \cocart(\ham(F^c))$ and $\ham(\cart(G^f)) \ra \cart(\ham(G^f))$ in $\Cat_{s\Set}$, but it is essentially no easier to show that these latter maps are weak equivalences in $(\Cat_{s\Set})_\Bergner$ than it is to prove \cref{main ingredient}.  However, these intermediate objects will appear in the course of the proof of \cref{main ingredient}.
\end{rem}

We end this section with a result used in the proof of the main theorem, which is to appear in the forthcoming paper \cite{BHH}.\footnote{The authors of \cite{BHH} in turn credit Cisinski for their proof.  They actually work in the more general setting of ``weak fibration categories'', and in their proof they replace the appeal to \cite[Theorem A.3.2(1)]{Hinich} with an appeal to work of Cisinski's.}

\begin{rem}\label{acknowledge DK}
The following result actually goes back to \cite[Proposition 5.2]{DKFunc}, but the proof given there relies on a claim whose proof is omitted, namely that the relative category $(\C^c,\bW^c_\C)$ admits a ``homotopy calculus of left fractions'' as in \cite[6.1(ii)]{DKCalc} (see \cite[8.2(ii)]{DKFunc}).  We have not been able to prove this result ourselves, so we provide this alternative proof for completeness.
\end{rem}

\begin{lem}\label{inclusion of co/fibrants induces BK weak equivalence}
The inclusions $(\C^c , \bW^c_\C) \hookra (\C,\bW_\C)$ and $(\D^f,\bW^f_\D) \hookra (\D,\bW_\D)$ are weak equivalences in $\RelCat_\BarKan$.
\end{lem}

\begin{proof}
We will prove the second of these two dual statements, which we will accomplish by showing that the map $(\D^f,\bW^f_\D) \hookra (\D,\bW_\D)$ induces a weak equivalence in $s(s\Set_\KQ)_\Reedy$ upon application of the functor $\NerveRezk : \RelCat \ra ss\Set$.  In other words, we will show that for all $n \geq 0$, the inclusion $\Fun([n],\D^f)^{\bW^f_\D} \hookra \Fun([n],\D)^{\bW_\D}$ induces a weak equivalence in $s\Set_\KQ$ upon application of the functor $\Nerve : \Cat \ra s\Set$.  For this, let us equip $\Fun([n],\D)$ with the projective model structure, which exists since it coincides with the Reedy model structure when we consider $[n]$ as a Reedy category with no non-identity degree-lowering maps.  Then, the above inclusion is precisely the inclusion $\bW^f_{\Fun([n],\D)} \hookra \bW_{\Fun([n],\D)}$, which induces a weak equivalence on nerves by combining the duals of \cite[Theorem A]{QuillenAKTI} and \cite[Theorem A.3.2(1)]{Hinich}.
\end{proof}

\section{Model diagrams and the proof of \cref{main ingredient}}\label{section proof of main ingredient}

In proving statements about categories of diagrams in model categories, it is convenient to have a general framework for parametrizing them.

\begin{defn}\label{define model diagram}
A \bit{model diagram} is a category $\I$ equipped with three wide subcategories $\bW_\I,\bC_\I,\bF_\I \subset \I$ such that $\bW_\I$ satisfies the two-out-of-three axiom.\footnote{The assumption that $\bW_\I$ satisfies the two-out-of-three axiom is probably superfluous, since we'll generally be mapping into model diagrams whose weak equivalences already have this property (namely the model category $\D$ as well as $(\C^c+\D^f)$ and its cousins).  Nevertheless, it seems like a good idea to include it, just to be safe.}  These assemble into the evident category, which we denote by $\Model$.  For $\I,\J \in \Model$, we denote by $\Fun(\I,\J)^\bW$ the category whose objects are morphisms of model diagrams and whose morphisms are natural \textit{weak equivalences} between them.  We will consider relative categories (and in particular, categories) as equipped with the \textit{minimal} model diagram structure (in which $\bC$ and $\bF$ consist only of the identity maps). % (Note that this notation $\Fun(\I,\J)^\bW$ is concordant with the notation given in \cref{section conventions}.)
\end{defn}

\begin{rem}
Among the axioms for a model category, all but the limit axiom (so the two-out-of-three, retract, lifting, and factorization axioms) can be encoded by requiring that the underlying model diagram has the extension property with respect to certain maps of model diagrams.
\end{rem}

\begin{var}\label{decorated model diagram}
A \bit{decorated model diagram} is a model diagram with some subdiagrams decorated as colimit or limit diagrams.\footnote{These are closely related to what are now called ``sketches'', originally introduced in \cite{EhresmannSketch}.}  For instance, if we define $\I$ to be the ``walking pullback square'', 
\begin{comment}
\[ \begin{tikzpicture}[node distance=2cm]
\node(a) {$\bullet$};
\node[right of=a] (b) {$\bullet$};
\node[below of=a] (c) {$\bullet$};
\node[right of=c] (d) {$\bullet$};
\node[below of=a, right of=a, node distance=0.5cm] (p) {};
\draw [style={-}, line join=miter] (p) to +(0.0,0.2) to +(0.0,0.0) to +(-0.2,0.0);
\draw[->] (a) -- (b);
\draw[->] (a) -- (c);
\draw[->] (b) -- (d);
\draw[->] (c) -- (d);
\end{tikzpicture} , \]
\end{comment}
then for any other model diagram $\J$, we let $\hom^\dec_\Model(\I,\J) \subset \hom_\Model(\I,\J)$ and $\Fun^\dec(\I,\J)^\bW \subset \Fun(\I,\J)^\bW$ denote the subobjects spanned by those morphisms $\I \ra \J$ of model diagrams which select a pullback square in $\J$.

For the most part, we will only use this variant on \cref{define model diagram} for pushout and pullback squares.  In fact, all but one of the model diagrams that we will decorate in this way will only have a single square anyways, and so in the interest of easing our Ti\textit{k}Zographical burden, we will simply superscript these model diagrams with either ``p.o.'' or ``p.b.'' (as in the proof of \cref{special-3 to 3} below).  The other one (which will appear in the proof of \cref{the two cases have homotopical three-arrow calculi}) will be endowed with sufficiently clear ad hoc notation.

However, this formalism also allows us to require that certain objects are sent to cofibrant or fibrant objects, by decorating a new object as initial/terminal and then marking its maps to/from the other objects as co/fibrations.  Rather than write this explicitly, we will abbreviate the notation for this procedure by superscripting objects by $c$, $f$, or $cf$ (to indicate that under a morphism of model diagrams they must select cofibrant, fibrant, or bifibrant objects, respectively).

Note that the constructions $\hom^\dec_\Model(\I,\J)$ and $\Fun^\dec(\I,\J)^\bW$ are not generally functorial in the target $\J$.  On the other hand, they are functorial for \textit{some} maps in the source $\I$.  We will refer to such maps as \bit{decoration-respecting}.  These define a category $\Model^\dec$.  (Note the distinction between $\hom_{\Model^\dec}$ and $\hom^\dec_\Model$.)  We consider $\Model \subset \Model^\dec$ simply by considering undecorated model diagrams as being trivially decorated.  We will not need a general theory for understanding which maps of decorated model diagrams are decoration-respecting; rather, it will suffice to observe once and for all
\begin{itemize}
\item that objects marked as co/fibrant must be sent to the same, and
\item that given a square which is decorated as a pushout or pullback square, it is decoration-respecting to either
\begin{itemize}
\item take it to another similarly decorated square, or
\item collapse it onto a single edge (since a square in which two parallel edges are identity maps is both a pushout and a pullback).
\end{itemize}
\end{itemize}

Note that if the source of a map of decorated model diagrams is actually undecorated, then the map is automatically decoration-respecting; in other words, we must only check that maps in which the \textit{source} is decorated are decoration-respecting.
\end{var}

\begin{rem}
Of course, adding in this variant allows us to also demand finite bicompleteness of a model diagram via lifting conditions, and hence all of the axioms for a model diagram to be a model category can now be encoded in this language.
\end{rem}

%%%%%

\begin{comment}

\begin{rem}
{\color{green} [make a further variant/remark for $(\C^c+\D^f)$.  set some convention having something to do with the ambient model category (either $\C$ or $\D$) for whichever side contains the cone point, instead of necessarily asking for a co/limit in the actual category $(\C^c+\D^f)$.  maybe label it as ``pushout in $\D$'' or whatever.]} {\color{magenta} note that pullbacks in $\D^f$ are preserved by the inclusion $\D^f \subset (\C^c+\D^f)$, and colimits in $\D$ are preserved by the inclusion $\D \subset (\C^c+\D)$.  dual statements apply for $\C^c$ and $\C$.  across-the-bridge pushouts in $(\C^c+\D)$ are computed by pushouts in $\D$.}  {\color{cyan} [where does this even appear? in the proof of 3.12(2) (for $(\C^c+\D)$); in the proof of 3.16(2) (but only sketchily).]}
\end{rem}

\end{comment}

%%%%%

We will be interested diagrams in model categories which connect specified ``source'' and ``target'' objects.  We thus introduce the following variant.

\begin{var}\label{define doubly-pointed model infty-diagram}
A \bit{doubly-pointed model diagram} is a model diagram $\I$ equipped with a map $\pt \sqcup \pt \ra \I$.  The two inclusions $\pt \hookra \pt \sqcup \pt$ select objects $s, t \in \I$, which we call the \bit{source} and the \bit{target}.  These assemble into the evident category, which we denote by $\Modelp = \Model_{\pt \sqcup \pt/}$.  Of course, there is a forgetful functor $\Modelp \ra \Model$, which we will occasionally implicitly use.  For $\I,\J \in \Modelp$, we denote by $\Funp(\I,\J)^{\bW} \subset \Fun(\I,\J)^\bW$ the (not generally full) subcategory whose objects are those morphisms of model diagrams which preserve the double-pointing, and whose morphisms are those natural weak equivalences whose components at $s$ and $t$ are respectively $\id_s$ and $\id_t$.  We will refer to such a morphism as a \bit{doubly-pointed natural weak equivalence}.  If we have chosen ``source'' and ``target'' objects in a model diagram, we will use these to consider the model diagram as doubly-pointed without explicitly mentioning it.  Of course, we may decorate a doubly-pointed model diagram as in \cref{decorated model diagram}.
\end{var}

\begin{var}
We define a \bit{model word} to be a word $\word{m}$ in any the symbols describing a morphism in a model diagram or their inverses (e.g.\! $\bW$, $(\bW \cap \bF)^{-1}$, $(\bW \cap \bC)$), or in the symbol $\any$ (for ``any arbitrary arrow'') or its inverse.  We will write $\any^{\circ n}$ to denote $n$ consecutive copies of the symbol $\any$ (for any $n \geq 0$).  We can extract a doubly-pointed model diagram from a model word, which for our sanity we will carry out by reading \textit{forwards}.  So for instance, the model word $\word{m} = [\bC;(\bW \cap \bF)^{-1};\any]$ defines the doubly-pointed model diagram
\[ \begin{tikzcd}
s \arrow[tail]{r} & \bullet & \bullet \arrow[two heads]{l}[swap]{\approx} \arrow{r} & t.
\end{tikzcd} \]
We denote this object again by $\word{m} \in \Modelp$.
\end{var}

\begin{rem}
Restricting to those model words in the symbols $\any$ and $\bW^{-1}$ and the \textit{order-preserving} maps between them, we recover the category of ``zigzag types'', i.e.\! the opposite of the category $\II$ of \cite[4.1]{DKCalc}.  In this way, we consider $\II^{op} \subset \Modelp$ as a (non-full) subcategory.  For any relative category $(\R,\bW_\R) \in \RelCat$ and any objects $x,y \in \R$, by \cite[Proposition 5.5]{DKCalc} we have an isomorphism
\[ \colim_{\word{m} \in \II} \Nerve\left( \Funp(\word{m},\R)^\bW \right)  \xra{\cong} \hom_{\ham(\R)}(x,y) \]
of simplicial sets, induced by the maps $\Nerve \left(\Funp(\word{m},\R)^\bW \right) \ra \hom_{\ham(\R)}(x,y)$ given by reducing the hammocks involved (as described in \cite[2.1]{DKCalc}).
\end{rem}

\begin{notn}
We will use the abbreviations $\word{3} = [\bW^{-1} ; \any ; \bW^{-1}]$ and $\tilde{\word{3}} = [(\bW \cap \bF)^{-1} ; \any ; (\bW \cap \bC)^{-1}]$; these model words correspond to the doubly-pointed model diagrams
\[ \begin{tikzcd}
s & \bullet \arrow{r} \arrow{l}[swap]{\approx} & \bullet & t \arrow{l}[swap]{\approx}
\end{tikzcd} \]
and
\[ \begin{tikzcd}
s & \bullet \arrow{r} \arrow[two heads]{l}[swap]{\approx} & \bullet & t , \arrow[tail]{l}[swap]{\approx}
\end{tikzcd} \]
respectively.
\end{notn}

Note that there is a unique map $\word{3} \ra \tilde{\word{3}}$ in $\Modelp$.  In \cite[7.2(ii)]{DKFunc}, Dwyer--Kan indicate how to prove that for any $x,y \in \D$, the induced composite
\[ \Nerve \left( \Funp(\tilde{\word{3}} , \D )^\bW \right) \ra \Nerve \left(\Funp(\word{3} , \D)^\bW \right) \ra \hom_{\ham(\D)}(x,y) \]
is a weak equivalence in $s\Set_\KQ$.  In order to prove \cref{main ingredient}, we will show that these two maps are again weak equivalences if we replace $\D$ by $(\C^c+\D^f)$ and take $x \in \C^c$ and $y \in \D^f$.  The arguments for $(\C^c+\D^f)$ are patterned on those for $\D$, and so for the sake of exposition we will re-prove that case in tandem.

We begin with some preliminary results.

\begin{lem}\label{abstractify DKFunc 8.1}
Choose any doubly-pointed model diagram $\I \in \Modelp$, select a weak equivalence in $\I$ by choosing a map $[\bW] \ra \I$ in $\Model$, and define $\J \in \Modelp$ by taking a pushout
\[ \begin{tikzcd}
{[\bW]} \arrow{r} \arrow{d} & \I \arrow{d} \\
{[(\bW \cap \bC) ; (\bW \cap \bF) ]} \arrow{r} & \J
\end{tikzcd} \]
in $\Model$ (where the left map is the unique map in $\Modelp$, and $\J$ is doubly-pointed via the composition $\pt \sqcup \pt \ra \I \ra \J$).  Then, the map $\I \ra \J$ induces
\begin{enumerate}

\item\label{abstractify for model cat} for any $x,y \in \D$, a weak equivalence
\[ \Nerve \left( \Funp ( \J , \D )^\bW \right)  \we \Nerve \left( \Funp ( \I , \D )^\bW \right) , \]
and

\item\label{abstractify for collage} for any $x,y \in (\C^c+\D^f)$, a weak equivalence
\[ \Nerve \left( \Funp ( \J , (\C^c+\D^f) )^\bW \right)  \we \Nerve \left( \Funp ( \I , (\C^c+\D^f) )^\bW \right) . \]

\end{enumerate}
\end{lem}

\begin{proof}
This is a mild generalization of \cite[8.1]{DKFunc}, and the proof adapts readily.\footnote{In \cite[6.7]{DKFunc}, which constructs (special) simplicial resolutions, the factorization of the latching-to-matching map which produces the next simplicial level should be as $\wcofibn \fibn$, not $\wcofibn \wfibn$.}  The following observations may be helpful.
\begin{itemize}
\item The proof is unaffected by whether or not the map $[\bW] \ra \I$ selects an identity map, and by whether or not it hits one or both of the objects $s,t\in \I$.
\item The proof does not require that the map $[\bW] \ra \I$ be ``free'' (i.e.\! obtained by taking a pushout $[\bW] \la \pt \sqcup \pt \ra \I'$), although we will actually only need this special case.
\item For item \ref{abstractify for collage}, note that all of the computations happen in one fiber or the other (since none of the bridge arrows are weak equivalences), and that the necessary simplicial resolution will automatically lie in the relevant subcategory $\C^c \subset \C$ or $\D^f \subset \D$ (in fact, it will even consist of \textit{bifibrant} objects).
\qedhere

\end{itemize}

\end{proof}

\begin{lem}\label{natural weak equivalences of model diagrams corepresent natural transformations}
Let $\I,\J \in \Modelpdec$ be decorated doubly-pointed model diagrams, let $\alpha,\beta : \I \rra \J$ be parallel morphisms in $\Modelpdec$, and let $\gamma : \alpha \ra \beta$ be a doubly-pointed natural weak equivalence.  Then for any $\E \in \Modelp$, $\gamma$ induces a natural transformation between the induced parallel maps $\alpha^*,\beta^* : \Funp^\dec(\J,\E)^\bW \rra \Funp^\dec(\I,\E)^\bW$ in $\Cat$.
\end{lem}

\begin{proof}
This is immediate from the definitions.
\end{proof}

We can now prove the first weak equivalence.

\begin{prop}\label{special-3 to 3}
The unique map $\word{3} \ra \tilde{\word{3}}$ in $\Modelp$ induces

\begin{enumerate}

\item\label{special-3 to 3 for model cat} for any $x,y \in \D$, a weak equivalence
\[ \Nerve \left( \Funp(\tilde{\word{3}} , \D )^\bW \right) \we \Nerve \left( \Funp(\word{3} , \D)^\bW \right) , \]
and

\item\label{special-3 to 3 for collage} for any $x \in \C^c$ and $y \in \D^f$, a weak equivalence
\[ \Nerve(\Funp(\tilde{\word{3}} , (\C^c+\D^f))^\bW) \we \Nerve(\Funp(\word{3} , (\C^c+\D^f))^\bW) . \]

\end{enumerate}
\end{prop}

\begin{proof}
We first address item \ref{special-3 to 3 for model cat}, which is somewhat simpler to prove.  For this, we factor the unique map $\word{3} \ra \tilde{\word{3}}$ in $\Modelp$ through a sequence
\[ \word{3} \xra{\varphi_1} \I_1 \xra{\varphi_2} \I_2 \xra{\varphi_3} \I_3 \xra{\varphi_4} \I_4 \xra{\varphi_5} \I_5 \xra{\varphi_6} \tilde{\word{3}} \]
of maps in $\Modelpdec$, given by
\[ \word{3} = \left( \begin{tikzcd}
s & \bullet \arrow{l}[swap, sloped]{\approx} \arrow{r} & \bullet & t \arrow{l}[swap]{\approx}
\end{tikzcd} \right)
\xra{\varphi_1} \left( \begin{tikzcd}[row sep=0.5cm]
& \bullet \arrow{ld}[swap, sloped, pos=0.3]{\approx} \arrow{r} \arrow[tail]{dd}[sloped, anchor=south]{\approx} & \bullet \\
s & & & t \arrow{lu}[sloped, swap, pos=0.7]{\approx} \\
& \bullet \arrow[two heads]{lu}[sloped, pos=0.3]{\approx}
\end{tikzcd} \right) \]
\[ \xra{\varphi_2} \left( \begin{tikzcd}[row sep=0.5cm]
& \bullet \arrow{ld}[swap, sloped, pos=0.3]{\approx} \arrow{r} \arrow[tail]{dd}[sloped, anchor=south]{\approx} & \bullet \arrow[tail]{dd}[sloped, anchor=north]{\approx} \\
s & & & t \arrow{lu}[sloped, swap, pos=0.7]{\approx} \arrow{ld}[sloped, pos=0.7]{\approx} \\
& \bullet \arrow{r} \arrow[two heads]{lu}[sloped, pos=0.3]{\approx} & \bullet
\end{tikzcd} \right)^{\textup{p.o.}}
\xra{\varphi_3} \left( \begin{tikzcd}
s & \bullet \arrow[two heads]{l}[swap]{\approx} \arrow{r} & \bullet & t \arrow{l}[swap]{\approx}
\end{tikzcd} \right)
\xra{\varphi_4} \left( \begin{tikzcd}[row sep=0.5cm]
& & \bullet \arrow[two heads]{dd}[sloped, anchor=north]{\approx} \\
s & & & t \arrow[tail]{lu}[sloped, swap, pos=0.7]{\approx} \arrow{ld}[sloped, pos=0.7]{\approx} \\
& \bullet \arrow{r} \arrow[two heads]{lu}[sloped, pos=0.3]{\approx} & \bullet
\end{tikzcd} \right) \]
\[ \xra{\varphi_5} \left( \begin{tikzcd}[row sep=0.5cm]
& \bullet \arrow[two heads]{ld}[swap, sloped, pos=0.3]{\approx} \arrow{r} \arrow[two heads]{dd}[sloped, anchor=south]{\approx} & \bullet \arrow[two heads]{dd}[sloped, anchor=north]{\approx} \\
s & & & t \arrow[tail]{lu}[sloped, swap, pos=0.7]{\approx} \arrow{ld}[sloped, pos=0.7]{\approx} \\
& \bullet \arrow{r} \arrow[two heads]{lu}[sloped, pos=0.3]{\approx} & \bullet
\end{tikzcd} \right)^{\textup{p.b.}}
\xra{\varphi_6} \left( \begin{tikzcd}
s & \bullet \arrow[two heads]{l}[swap, sloped]{\approx} \arrow{r} & \bullet & t \arrow[tail]{l}[swap]{\approx}
\end{tikzcd} \right) = \tilde{\word{3}}, \]
in which
\begin{itemizesmall}
\item the maps $\varphi_1$, $\varphi_2$, $\varphi_4$, and $\varphi_5$ are the evident inclusions, and
\item the maps $\varphi_3$ and $\varphi_6$ are given by collapsing vertically.
\end{itemizesmall}

We now prove that each of these maps induces a weak equivalence upon application of $\Nerve\left(\Funp^\dec(-,\D)^\bW\right)$.  The arguments can be grouped as follows.
\begin{itemize}

\item The fact that the maps $\varphi_1$ and $\varphi_4$ induce weak equivalences follows from \cref{abstractify DKFunc 8.1}\ref{abstractify for model cat}.

\item The maps $\varphi_2$ and $\varphi_5$ induce acyclic fibrations in $s\Set_\KQ$, since
\begin{itemize}
\item $\D$ is finitely bicomplete,
\item limits and colimits are unique up to unique isomorphism, and
\item the subcategories $(\bW \cap \bC)_\D, (\bW \cap \bF)_\D \subset \D$ are respectively closed under pushout and pullback
\end{itemize}
(see e.g.\! \cite[Proposition 4.3.2.15]{LurieHTT}).

\item Note that the maps $\varphi_3$ and $\varphi_6$ admit respective sections $\psi_3$ and $\psi_6$ in $\Modelpdec$.  Moreover, there are evident doubly-pointed natural weak equivalences $\id_{\I_2} \ra \psi_3 \varphi_3$ and $\psi_6 \varphi_6 \ra \id_{\I_5}$.  Hence, it follows from \cref{natural weak equivalences of model diagrams corepresent natural transformations} that $\varphi_3$ and $\varphi_6$ induce homotopy equivalences in $s\Set_\KQ$.

\end{itemize}

The proof of item \ref{special-3 to 3 for collage} is similar, but requires some modification: we will show that the above sequence induces weak equivalences in $s\Set_\KQ$ upon application of $\Nerve \left( \Funp ( - , (\C^c+\D^f) )^\bW \right)$ (note the lack of decorations).\footnote{In fact, the following argument also works for $\D$, but its additional complexity would just have been confusingly unnecessary if we had provided it above.}

First of all, note that since we have assumed that $x \in \C^c$ and $y \in \D^f$, then all maps to $(\C^c+\D^f)$ in $\Modelp$ from all doubly-pointed model diagrams in the above sequence must have that the unmarked arrows select bridge arrows, with everything to the left mapping into $\C^c$ and everything to the right mapping into $\D^f$.

We now show that restriction along $\varphi_2$ induces a weak equivalence
\[ \Nerve \left( \Funp ( \I_2 , (\C^c+\D^f) )^\bW \right) \we \Nerve \left( \Funp (\I_1 , (\C^c+\D^f))^\bW \right) . \]
For this, we use the analogously defined object $(\C^c+\D) \in \Modelp$, and we define a diagram
\[ \I_{2a} \xra{\kappa_1} \I_{2b} \xla{\kappa_2} \I_{2c} \xra{\kappa_3} \I_{2d} \xla{\kappa_4} \I_{2e} \]
in $\Modelpdec$, given by
\[ \left( \begin{tikzcd}[row sep=0.5cm]
& \bullet \arrow{ld}[swap, sloped, pos=0.3]{\approx} \arrow{r} \arrow[tail]{dd}[sloped, anchor=south]{\approx} & \bullet^f \\
s & & & t \arrow{lu}[sloped, swap, pos=0.7]{\approx} \\
& \bullet \arrow[two heads]{lu}[sloped, pos=0.3]{\approx}
\end{tikzcd} \right)
\xra{\kappa_1} \left( \begin{tikzcd}[row sep=0.5cm]
& \bullet \arrow{ld}[swap, sloped, pos=0.3]{\approx} \arrow{r} \arrow[tail]{dd}[sloped, anchor=south]{\approx} & \bullet^f \arrow[tail]{dd}[sloped, anchor=north]{\approx} \\
s & & & t \arrow{lu}[sloped, swap, pos=0.7]{\approx} \arrow{ld}[sloped, pos=0.7]{\approx} \\
& \bullet \arrow{r} \arrow[two heads]{lu}[sloped, pos=0.3]{\approx} & \bullet
\end{tikzcd} \right)^{\textup{p.o.}} 
\xla{\kappa_2} \left( \begin{tikzcd}[row sep=0.5cm]
& \bullet \arrow{ld}[swap, sloped, pos=0.3]{\approx} \arrow{r} \arrow[tail]{dd}[sloped, anchor=south]{\approx} & \bullet^f \arrow[tail]{dd}[sloped, anchor=north]{\approx} \\
s & & & t \arrow{lu}[sloped, swap, pos=0.7]{\approx} \arrow{ld}[sloped, pos=0.7]{\approx} \\
& \bullet \arrow{r} \arrow[two heads]{lu}[sloped, pos=0.3]{\approx} & \bullet
\end{tikzcd} \right) \]
\[ \xra{\kappa_3} \left( \begin{tikzcd}[row sep=0.5cm]
& \bullet \arrow{ld}[swap, sloped, pos=0.3]{\approx} \arrow{r} \arrow[tail]{dd}[sloped, anchor=south]{\approx} & \bullet^f \arrow[tail]{dd}[sloped, anchor=north]{\approx} \\
s & & & t \arrow{lu}[sloped, swap, pos=0.7]{\approx} \arrow{ld}[sloped, pos=0.7]{\approx} \arrow{lddd}[sloped, pos=0.6]{\approx} \\
& \bullet \arrow{r} \arrow{rdd} \arrow[two heads]{lu}[sloped, pos=0.3]{\approx} & \bullet \arrow[tail]{dd}[sloped, anchor=north]{\approx} \\ \\
& & \bullet^f
\end{tikzcd} \right)
\xla{\kappa_4} \left( \begin{tikzcd}[row sep=0.5cm]
& \bullet \arrow{ld}[swap, sloped, pos=0.3]{\approx} \arrow{r} \arrow[tail]{dd}[sloped, anchor=south]{\approx} & \bullet^f \arrow[tail]{dddd}[sloped, anchor=north]{\approx} \\
s & & & t \arrow{lu}[sloped, swap, pos=0.7]{\approx} \arrow{lddd}[sloped, pos=0.6]{\approx} \\
& \bullet \arrow{rdd} \arrow[two heads]{lu}[sloped, pos=0.3]{\approx} \\ \\
& & \bullet^f
\end{tikzcd} \right) , \]
in which all maps are the evident inclusions.  Then, we proceed by the following arguments.
\begin{itemize}

\item There is an evident isomorphism
\[ \Funp(\I_1 , (\C^c+\D^f) )^\bW \cong \Funpdec ( \I_{2a} , (\C^c+\D))^\bW . \]

\item The map $\kappa_1$ induces an acyclic fibration 
\[ \Nerve \left( \Funpdec ( \I_{2b} , (\C^c+\D))^\bW \right)  \wfibn \Nerve \left( \Funpdec ( \I_{2a} , (\C^c+\D))^\bW \right) \]
in $s\Set_\KQ$ for the same reasons that $\varphi_2$ and $\varphi_5$ induced them above.

\item The map $\kappa_2$ induces a weak equivalence
\[ \Nerve \left( \Funpdec(\I_{2b} , (\C^c+\D))^\bW \right) \we \Nerve \left( \Funpdec(\I_{2c},(\C^c+\D))^\bW \right) \]
in $s\Set_\KQ$ since it is the nerve of a functor which admits a right adjoint.  (Using the dual of the characterization of \cite[Chapter IV, \sec 1, Theorem 2(ii)]{MacLaneCWM}, to obtain such a right adjoint suffices to choose a coreflection in $\Funpdec(\I_{2b} , (\C^c+\D))^\bW$ of each object of $\Funpdec(\I_{2c} , (\C^c+\D))^\bW$, along with the corresponding component of the counit: to obtain a coreflection we can take a pushout of the span defined by our object, and the corresponding component of the counit will then be the canonical map.)

\item The map $\kappa_3$ induces a weak equivalence
\[ \Nerve \left( \Funpdec ( \I_{2d} , (\C^c+\D))^\bW \right) \we \Nerve \left( \Funpdec ( \I_{2c} , (\C^c+\D))^\bW \right) \]
by \cref{adding a fibrant replacement by an acyclic cofibration to a model diagram gives a weak equivalence in collage} below.

\item The inclusion $\kappa_4$ admits a retraction $\lambda_4$, given by taking the additional object in $\I_{2d}$ to the bottommost object in $\I_{2e}$.  Moreover, there is an evident doubly-pointed natural weak equivalence $\id_{\I_{2d}} \ra \kappa_4 \lambda_4$.  Hence, it follows from \cref{natural weak equivalences of model diagrams corepresent natural transformations} that $\kappa_4$ induces a homotopy equivalence
\[ \Nerve \left( \Funpdec ( \I_{2d} , (\C^c+\D))^\bW \right) \we \Nerve \left( \Funpdec ( \I_{2e} , (\C^c+\D))^\bW \right) \]
in $s\Set_\KQ$.

\item There is an evident isomorphism
\[ \Funpdec (\I_{2e} , (\C^c+\D))^\bW \cong \Funp(\I_2 , (\C^c+\D^f))^\bW . \]

\end{itemize}
As these weak equivalences are compatible with the map
\[ \Nerve \left( \Funp ( \I_2 , (\C^c+\D^f) )^\bW \right) \ra \Nerve \left( \Funp (\I_1 , (\C^c+\D^f))^\bW \right) \]
induced by restriction along $\varphi_2$ (in the sense that adding in the evident inclusion $\I_{2a} \ra \I_{2e}$ in $\Modelpdec$, to which when we apply $\Nerve \left( \Funpdec ( - , (\C^c+\D))^\bW \right)$ yields an isomorphic map to the one given by applying $\Nerve \left( \Funp ( - , (\C^c+\D^f))^\bW \right)$ to $\I_1 \xra{\varphi_2} \I_2$, yields a commutative diagram in $\Modelp$), we see that it is indeed a weak equivalence as well.

From here, a nearly identical dual argument (this time using $(\C+\D^f) \in \Modelp$) shows that restriction along $\varphi_5$ also induces a weak equivalence
\[ \Nerve \left( \Funp ( \I_5 , (\C^c+\D^f) )^\bW \right) \we \Nerve \left( \Funp (\I_4 , (\C^c+\D^f))^\bW \right) \]
So we have proved that upon application of $\Nerve\left(\Funp(-,(\C^c+\D^f))^\bW \right)$, the maps $\varphi_2$ and $\varphi_5$ induce weak equivalences in $s\Set_\KQ$.  That $\varphi_1$ and $\varphi_4$ induce weak equivalences follows from \cref{abstractify DKFunc 8.1}\ref{abstractify for collage}, and that $\varphi_3$ and $\varphi_6$ induce weak equivalences follows from the same argument as given above in the proof of item \ref{special-3 to 3 for model cat}.
\end{proof}

We now prove a lemma that was used in the proof of \cref{special-3 to 3}.

\begin{lem}\label{adding a fibrant replacement by an acyclic cofibration to a model diagram gives a weak equivalence in collage}
Choose any doubly-pointed model diagram $\I \in \Modelp$, choose an object $a \in \I$ which is connected by a zigzag in $\bW_\I$ to the object $t \in \I$, and use it to define $\J \in \Modelp$ by freely adjoining a new fibrant object and an acyclic cofibration $a \wcofibn \bullet^f$ to it from $a$.\footnote{That is, if $\I$ already has an object marked as terminal then we add a new object equipped with a fibration to it; otherwise we add both.}  Then, assuming that the target object of $(\C^c+\D)$ lives in the subcategory $\D \subset (\C^c+\D)$, the evident inclusion $\I \xra{\varphi} \J$ induces a weak equivalence
\[ \Nerve \left( \Funpdec ( \J , (\C^c+\D) )^\bW \right) \xra{\approx} \Nerve \left( \Funpdec ( \I , (\C^c+\D) )^\bW \right) \]
in $s\Set_\KQ$.
\end{lem}

\begin{proof}
To ease notation, let us write this map as $\Nerve(\mathbf{B}_1) \ra \Nerve(\mathbf{B}_2)$.  Then, appealing to the dual of \cite[Theorem A]{QuillenAKTI}, it suffices to prove that for any $b \in \mathbf{B}_2$, the comma category
\[ \mathbf{B}_3 = \mathbf{B}_1 \times_{\mathbf{B}_2} (\mathbf{B}_2)_{b/} \]
has weakly contractible nerve.

For this, define the subcategory $\mathbf{B}_3' \subset \mathbf{B}_3$ on those objects $(c , b \ra \varphi^*(c)) \in \mathbf{B}_3$ such that the doubly-pointed natural weak equivalence $b \ra \varphi^*(c)$ is actually $\id_b$.  Then $\mathbf{B}_3'$ is isomorphic to the category
\[ (\bW^f_\D)_{b(a) \dcofibn} = \bW^f_\D \times_{\bW_\D} (\bW_\D)_{b(a) \dcofibn} \]
 of acyclic cofibrations from $b(a)$ to a fibrant object, which has weakly contractible nerve by applying the dual of \cite[Theorem A]{QuillenAKTI} to \cref{similar to DKFunc 6.11} (since $\Nerve(\bD^{op})$ is weakly contractible).\footnote{This statement also follows from applying the dual of \cite[Theorem A.3.2(2)]{Hinich} to the initial object of the model category $\D_{b(a)/}$.}  On the other hand, to show that the map $\Nerve(\mathbf{B}_3') \ra \Nerve(\mathbf{B}_3)$ is a weak equivalence, again appealing to \cite[Theorem A]{QuillenAKTI}, it suffices to prove that for any $(c,b \ra \varphi^*(c)) \in \mathbf{B}_3$, the comma category
\[ \mathbf{B}_4 = \mathbf{B}_3' \times_{\mathbf{B}_3} (\mathbf{B}_3)_{/(c, b \ra \varphi^*(c))} \]
has weakly contractible nerve.

Now, an object of $\mathbf{B}_4$ is given by the pair of an object $(c' , b \xra{=} \varphi^*(c')) \in \mathbf{B}_3'$ and a morphism $(c' , b \xra{=} \varphi^*(c')) \ra (c , b\ra \varphi^*(c))$ in $\mathbf{B}_3$.  Unwinding the definitions, we see that the data of such an object is precisely that of a factorization of the composite $b(a) \xra{=} c'(a) \we c(a) \wcofibn c(\bullet^f)$ in $\D \subset (\C^c+\D)$ through some composite $b(a) \xra{=} c'(a) \wcofibn c'(\bullet^f) \we c(\bullet^f)$, i.e.\! the specification of the upper right composite in a commutative square
\[ \begin{tikzcd}
b(a) \arrow[tail]{r}{\approx} \arrow{d}[sloped, anchor=north]{\approx} & c'(\bullet^f) \arrow{d}[sloped, anchor=south]{\approx} \\
c(a) \arrow[tail]{r}[swap]{\approx} & c(\bullet^f)
\end{tikzcd} \]
in $\D$.  Hence the category $\mathbf{B}_4$ is isomorphic to the category
\[ \bW^{cf}_{\D_{b(a)/}} \times_{\bW_{\D_{b(a)/}}} (\bW_{\D_{b(a)/}})_{/c(\bullet^f)} \]
of (left) replacements of the fibrant object $c(\bullet^f) \in \D_{b(a)/}$ by a bifibrant object, which has weakly contractible nerve by \cite[Theorem A.3.2(2)]{Hinich}.
\end{proof}

We now prove a lemma that was used in the proof of \cref{adding a fibrant replacement by an acyclic cofibration to a model diagram gives a weak equivalence in collage}.

\begin{lem}\label{similar to DKFunc 6.11}
Any object $y \in \D$ admits a special simplicial replacement $y_\bullet \in s(\bW^f_\D)$ (as in \cite[4.3 and Remark 6.8]{DKFunc}), and for any such choice, the corresponding map $\bD^{op} \xra{y_\bullet} (\bW^f_\D)_{y \dcofibn}$ is homotopy right cofinal.
\end{lem}

\begin{proof}
The first statement is just \cite[Proposition 4.5 and 6.7]{DKFunc}.  The second statement follows from combining \cite[Proposition 6.11]{DKFunc} with the following general fact: if a composite of functors is homotopy right cofinal and the second functor is fully faithful, then the first functor is also homotopy right cofinal.
\end{proof}

We now move on towards proving the second weak equivalence.  Purely as a matter of terminology, we begin with a slight variant on \cite[Definition 3.1]{LowMG}, which is in turn a slight variant on the original definition of a ``homotopy calculus of fractions'' given in \cite[6.1(i)]{DKCalc}.

\begin{defn}\label{define homotopical three-arrow calculus}
Let $(\R,\bW_\R)$ be a relative category, and let $x,y \in \R$.  We say that the relative category $(\R,\bW_\R)$ admits a \bit{homotopical three-arrow calculus} with respect to $x$ and $y$ if for all $i,j \geq 1$, the evident map
\[ \Nerve \left( \Funp ( [\bW^{-1} ; \any^{\circ i} ; \any^{\circ j} ; \bW^{-1}] , \R )^\bW \right) \ra \Nerve \left( \Funp ( [\bW^{-1} ; \any^{\circ i} ; \bW^{-1} ; \any^{\circ j} ; \bW^{-1}] , \R )^\bW \right) \]
is a weak equivalence in $s\Set_\KQ$.
\end{defn}

This notion is useful for the following reason (which is where it gets its name).

\begin{prop}\label{homotopical three-arrow calculus helps compute hammock sset}
If $(\R,\bW_\R)$ admits a homotopical three-arrow calculus with respect to $x$ and $y$, then the map
\[ \Nerve \left( \Funp ( \word{3} , \R)^\bW \right) \ra \hom_{\ham(\R)}(x,y) \]
is a weak equivalence in $s\Set_\KQ$.
\end{prop}

\begin{proof}
This is essentially \cite[Proposition 6.2]{DKCalc}; that the proof carries over to the present setting is justified in \cite[Theorem 3.3(ii)]{LowMG}.\footnote{A few minor typos in the proof of \cite[Proposition 6.2]{DKCalc} are also corrected in the proof of \cite[Theorem 3.3(ii)]{LowMG}.}
\end{proof}

Using this language, we can now prove (a result which in light of \cref{homotopical three-arrow calculus helps compute hammock sset} will directly imply) the second weak equivalence.

\begin{prop}\label{the two cases have homotopical three-arrow calculi}
The doubly-pointed relative categories
\begin{enumerate}
\item\label{model cat has homotopical three-arrow calculus} $(\D,\bW_\D)$, where $x,y \in \D$ are any objects, and
\item\label{collage has homotopical three-arrow calculus} $(\C^c+\D^f,\bW_{\C^c+\D^f})$, where $x \in \C^c$ and $y \in \D^f$
\end{enumerate}
admit homotopical three-arrow calculi.
\end{prop}

\begin{proof}
Again, we begin with item \ref{model cat has homotopical three-arrow calculus} since it is somewhat simpler to prove.  In this case, we define a diagram
\[ [\bW^{-1} ; \any^{\circ i} ; \bW^{-1} ; \any^{\circ j} ; \bW^{-1}] \xra{\rho_1} \J_1 \xra{\rho_2} \J_2 \xla{\rho_3} [\bW^{-1} ; \any^{\circ i} ; \any^{\circ j} ; \bW^{-1}] \]
in $\Modelpdec$, given by the evident inclusions
\[ [\bW^{-1} ; \any^{\circ i} ; \bW^{-1} ; \any^{\circ j} ; \bW^{-1}] = \left( \begin{tikzcd}
s \arrow[-]{rr}{[\bW^{-1} ; \any^{\circ i}]} & & \bullet & \bullet \arrow{l}[swap]{\approx} \arrow[-]{rr}{[\any^{\circ j} ; \bW^{-1}]} & & t
\end{tikzcd} \right) \]
\[ \xra{\rho_1} \left( \begin{tikzcd}[column sep=0.5cm]
s \arrow[-]{rrrr}{[\bW^{-1} ; \any^{\circ i}]} & & & & \bullet& & \bullet \arrow{ll}[swap]{\approx} \arrow[-]{rrrr}{[\any^{\circ j} ; \bW^{-1}]} \arrow[tail]{ld}[sloped, pos=0.7]{\approx} & & & & t \\
& & & & & \bullet \arrow[two heads]{lu}[sloped, pos=0.3]{\approx}
\end{tikzcd} \right) 
\]
\[ \xra{\rho_2} \left( \begin{tikzcd}[column sep=0.5cm]
s & & \bullet \arrow{ll}[swap]{\approx} \arrow{rr} & & \cdots \arrow{rr} & & \bullet & & \bullet \arrow{ll}[swap]{\approx} \arrow{rr} \arrow[tail]{ld}[sloped, pos=0.7]{\approx} & & \cdots \arrow{rr} & & \bullet \arrow[tail]{ld}[sloped, swap, pos=0.4]{\approx} & & t \arrow{ll}[swap]{\approx} \arrow[bend left=10]{llld}[sloped, pos=0.6]{\approx} \\
& & & \bullet \arrow{rr} \arrow[bend left=10]{lllu}[sloped, pos=0.4]{\approx} \arrow[two heads]{lu}[sloped, swap, pos=0.6]{\approx} & & \cdots \arrow{rr} & & \bullet \arrow{rr} \arrow[two heads]{lu}[sloped, pos=0.3]{\approx} & & \cdots \arrow{rr} & & \bullet
\end{tikzcd} \right)^{i \textup{ p.b.'s, } j \textup{ p.o.'s}} \]
\[ \xla{\rho_3} \left( \begin{tikzcd}[column sep=0.5cm]
s & & & & & & & & & & & & & & t \arrow[bend left=10]{llld}[sloped, pos=0.6]{\approx} \\
& & & \bullet \arrow{rr} \arrow[bend left=10]{lllu}[sloped, pos=0.4]{\approx} & & \cdots \arrow{rr} & & \bullet \arrow{rr} & & \cdots \arrow{rr} & & \bullet
\end{tikzcd} \right) =  [\bW^{-1} ; \any^{\circ i} ; \any^{\circ j} ; \bW^{-1}] . \]
We claim that this induces a diagram of weak equivalences in $s\Set_\KQ$ upon application of $\Nerve \left( \Funpdec ( - , \D)^\bW \right)$ in a way compatible with the map in \cref{define homotopical three-arrow calculus} (in the sense that adding in the evident map which corepresents it gives a commutative diagram in $\Modelpdec$).  We proceed by the following arguments.
\begin{itemize}

\item The map $\rho_1$ induces a weak equivalence by \cref{abstractify DKFunc 8.1}\ref{abstractify for model cat}.

\item The map $\rho_2$ induces an acyclic fibration in $s\Set_\KQ$ by repeatedly applying the argument for why the maps $\varphi_2$ and $\varphi_5$ induce them in the proof of \cref{special-3 to 3}\ref{special-3 to 3 for model cat}.

\item The inclusion $\rho_3$ admits a retraction $\sigma_3$ in $\Modelpdec$, given by collapsing the whole top row of $\J_2$ (besides the objects $s$ and $t$) down onto the lower row, so that the ``middle'' backwards weak equivalence in the top row gets sent to the identity map on the ``middle'' object of the bottom row.  If we define $\tau_3 \in \hom_{\Modelpdec}(\J_2,\J_2)$ to be the morphism which collapses the ``left half'' of $\J_2$ (besides the object $s$) down onto the lower row and leaves the ``right half'' unchanged, then there is an evident span
\[ \id_{\J_2} \la \tau_3 \ra \rho_3 \sigma_3 \]
of doubly-pointed natural weak equivalences.  Hence, it follows from \cref{natural weak equivalences of model diagrams corepresent natural transformations} that $\rho_3$ induces a homotopy equivalence in $s\Set_\KQ$.

\end{itemize}
Thus, $\D$ admits a homotopical three-arrow calculus (with respect to any double-pointing).

For item \ref{collage has homotopical three-arrow calculus}, as in the proof of \cref{special-3 to 3}\ref{special-3 to 3 for collage} we again modify the proof of item \ref{model cat has homotopical three-arrow calculus} by applying the functor $\Nerve \left( \Funp(- , (\C^c+\D^f))^\bW \right)$ to the given diagram (i.e.\! by ignoring decorations).

The first thing to observe here is that since no bridge arrows are weak equivalences, then the path components of both the source and the target of
\[ \Nerve \left( \Funp ( [\bW^{-1} ; \any^{\circ i} ; \any^{\circ j} ; \bW^{-1}] , (\C^c+\D^f) )^\bW \right) \ra \Nerve \left( \Funp ( [\bW^{-1} ; \any^{\circ i} ; \bW^{-1} ; \any^{\circ j} ; \bW^{-1}] , (\C^c+\D^f) )^\bW \right) \]
decompose according to where the (necessarily unique) bridge arrow lies among the $(i+j)$ possibilities, and moreover the map respects these decompositions.  For each $k \in \{ 1, \ldots, i+j\}$, let us use the ad hoc notation $\Funp^{\{k\}}(-,(\C^c+\D^f))^\bW$ to denote the respective subcategories on those zigzags where the $k\th$ copy of $\any$ gets sent to a bridge arrow; it suffices to show that the functor $\Nerve \left( \Funp^{\{k\}}(-,(\C^c+\D^f))^\bW \right)$ induces a weak equivalence in $s\Set_\KQ$ for each $k$.  We will focus on the case that $k \leq i$; the case that $k \geq i+1$ will follow from a completely dual argument.

\begin{itemize}

\item The fact that $\rho_1$ induces a weak equivalence in $s\Set_\KQ$ follows from \cref{abstractify DKFunc 8.1}\ref{abstractify for collage}.

\item To see that $\rho_2$ induces a weak equivalence, we will recycle arguments from the proof of \cref{special-3 to 3}\ref{special-3 to 3 for collage}.  Unlike in the proof of item \ref{model cat has homotopical three-arrow calculus}, we will need to separate the multiple steps in which we build the map $\J_1 \xra{\rho_2} \J_2$ into various cases.  Note that by our assumption that $k \leq i$, the objects of $\Funp^{\{k\}}(\J_1,(\C^c+\D^f))^\bW$ all select maps $\J_1 \ra (\C^c+\D^f)$ such that the composite in $\Model$ with the evident inclusion $[(\bW \cap \bF)^{-1} ; (\bW \cap \bC)^{-1}] \ra \J_1 \ra (\C^c+\D^f)$ lands in $\D^f \subset (\C^c+\D^f)$.

\begin{itemize}

\item We begin by adding the new acyclic cofibrations one by one, moving to the right and working in $\D^f \subset (\C^c+\D^f)$.  That these induce weak equivalences in $s\Set_\KQ$ upon applying $\Nerve \left( \Funp ( - , (\C^c+\D^f) )^\bW \right)$ follows from a nearly identical argument to the one that the same functor induces one upon application to $\varphi_2$.

\item Then, we begin adding the new acyclic fibrations, moving to the left and working in $\D^f \subset (\C^c+\D^f)$ until we reach the bridge arrow.  In this case, we can work entirely within $(\C^c+\D^f)$ (i.e.\! without using the auxiliary object $(\C^c+\D) \in \Modelp$) since the pullback of an acyclic fibration among fibrant objects will automatically be fibrant, and we can use completely dual arguments to the ones used to show that $\kappa_1$ and $\kappa_2$ induce weak equivalences to see that this again induces a weak equivalence.

\item Once we hit the bridge arrow and thereafter, we continue adding the new acyclic fibrations.  But now, since we will be adding objects that are in $\C^c$ that we will originally construct via pullback in $\C$, we use the full strength of the argument that the functor $\Nerve \left( \Funp ( - , (\C^c+\D^f) )^\bW \right)$ induces a weak equivalence in $s\Set_\KQ$ upon application to $\varphi_5$.

\end{itemize}

\item The fact that $\rho_3$ induces a weak equivalence follows from the same argument as given above in the proof of item \ref{model cat has homotopical three-arrow calculus}.

\end{itemize}
Thus, $(\C^c+\D^f)$ admits a homotopical three-arrow calculus with respect to any $x \in \C^c$ and $y \in \D^f$.
\end{proof}

We can now give a proof of the main ingredient in the proof of the main theorem, which is based on the arguments of \cite[\sec 7]{Mandell}.

\begin{proof}[Proof of \cref{main ingredient}]
We will prove that the map $\ham(\C^c+\D^f) \ra \cocart(\ham(F^c))$ is a weak equivalence in $((\Cat_{s\Set})_{\Bergner})_{/[1]}$; that the map $\ham(\C^c+\D^f) \ra \cart(\ham(G^f))$ is also a weak equivalence will follow from a completely dual argument.

First of all, over $0 \in [1]$ this map is an isomorphism $\ham(\C^c) \xra{\cong} \ham(\C^c)$, while over $1 \in [1]$ this map is given by the inclusion $\ham(\D^f) \hookra \ham(\D)$, which is a weak equivalence by \cref{inclusion of co/fibrants induces BK weak equivalence}.  We claim that for any $x \in \C^c$ and any $y \in \D^f$, the induced map
\[ \hom_{\ham(\C^c+\D^f)}(x,y) \ra \hom_{\cocart(\ham(F^c))}(x,y) = \hom_{\ham(\D)}(F^c(x),y) \]
is a weak equivalence in $s\Set_\KQ$.  From here, it will follow easily that the induced map on homotopy categories (i.e.\! under the functor $\ho : \Cat_{s\Set} \ra \Cat$, given locally by the product-preserving functor $\pi_0 : s\Set \ra \Set$) will be an equivalence of categories with target the analogously defined object $\cocart(\ho(F^c)) \in \Cat_{/[1]}$.  Hence, we will have shown that the map $\ham(\C^c+\D^f) \ra \cocart(\ham(F^c))$ is indeed a weak equivalence in $((\Cat_{s\Set})_{\Bergner})_{/[1]}$.

So, let $x \in \C^c$ and $y \in \D^f$.  Let $x^\bullet \in c(\bW^c_\C)$ be a special cosimplicial resolution of $x$, and let $y_\bullet \in s (\bW^f_\D )$ be a special simplicial resolution of $y$ (as in \cite[4.3 and Remark 6.8]{DKFunc}; the existence of these resolutions is guaranteed by \cite[Proposition 4.5 and 6.7]{DKFunc}).  Let us also define a simplicial set $M_\bullet \in s\Set$ with
\[ M_n = \coprod_{(\alpha,\beta) \in \hom_\Cat([n],\bD \times \bD^{op})} \hom_{(\C^c+\D^f)}(x^{\alpha(n)},y_{\beta(0)})
\cong
\coprod_{(\alpha,\beta) \in \hom_\Cat([n],\bD \times \bD^{op})} \hom_\D(F^c(x^{\alpha(n)}),y_{\beta(0)}) \]
and with structure maps as in \cite[\sec 7]{Mandell}.  Then, considering $(\C^c+\D^f) \in \Modelp$ via the objects $x,y \in (\C^c+\D^f)$ and considering $\D \in \Modelp$ via the objects $F^c(x),y \in \D$, we claim that we obtain a commutative diagram
\[ \begin{tikzcd}
\diag\left( \hom^\lw_{(\C^c+\D^f)}(x^\bullet,y_\bullet) \right) \arrow{r}{\cong} & \diag \left( \hom^\lw_{\D}(F^c(x^\bullet),y_\bullet) \right) \\
M_\bullet \arrow{r}{=} \arrow{u}[sloped, anchor=south]{\approx} \arrow{d}[sloped, anchor=north]{\approx} & M_\bullet \arrow{u}[sloped, anchor=north]{\approx} \arrow{d}[sloped, anchor=south]{\approx} \\
\Nerve \left( \Funp(\tilde{\word{3}} , (\C^c+\D^f))^\bW \right) \arrow{r} \arrow{d}[sloped, anchor=north]{\approx} & \Nerve \left( \Funp(\tilde{\word{3}} , \D)^\bW \right) \arrow{d}[sloped, anchor=south]{\approx} \\
\Nerve \left( \Funp(\word{3} , (\C^c+\D^f))^\bW \right) \arrow{r} \arrow{d}[sloped, anchor=north]{\approx} & \Nerve \left( \Funp(\word{3} , \D)^\bW \right) \arrow{d}[sloped, anchor=south]{\approx} \\
\hom_{\ham(\C^c+\D^f)}(x,y) \arrow{r} & \hom_{\ham(\D)}(F^c(x),y)
\end{tikzcd} \]
in $s\Set_\KQ$ (where the maps involving $M_\bullet$ are as described in \cite[\sec 7]{Mandell}), from which it will follow that the bottom map is indeed a weak equivalence as well.\footnote{There is a small mistake in the description of the first pair of vertical arrows in \cite[sec 7]{Mandell}: in the notation there, the map $f \in \hom_\bD([m],[p_m])$ should be given by $i \mapsto f_m \circ \cdots \circ f_{i+1}(p_i)$ for $i<m$ and $m \mapsto p_m$, and the map $g \in \hom_\bD([m],[q_0])$ should be given by $0 \mapsto 0$ and $i \mapsto g_1 \circ \cdots \circ g_i(0)$ for $i>0$.}  We argue as follows.
\begin{itemize}

\item The first right vertical arrow is a weak equivalence by \cite[Proposition 7.2]{Mandell}, and hence (using the evident fact that the top horizontal map is indeed an isomorphism) we obtain that the first pair of vertical arrows are weak equivalences.\footnote{Of course, the uppermost square is unnecessary from a strictly logical point of view, but it clarifies the connection between our proof and co/simplicial resolutions.}

\item The second right vertical arrow is a weak equivalence by \cite[Proposition 7.3]{Mandell}.  The second left vertical arrow is a weak equivalence by a similar argument; we modify the one given there as follows.

\begin{itemize}

\item We redefine $N_\bullet \in s\Set$ to be the simplicial replacement of the functor
\[ \left( ( \bW^c_\C )_{\dfibn x} \right)^{op} \times ( \bW^f_\D )_{y \dcofibn} \xra{\hom_{(\C^c+\D^f)}(-,-)} \Set . \]

\item The functor $\bD^{op} \xra{y_\bullet} (\bW^f_\D)_{y \dcofibn}$ is again homotopy right cofinal by \cref{similar to DKFunc 6.11}; the functor $\bD \xra{x^\bullet} (\bW^c_\C)_{\dfibn x}$ is again homotopy left cofinal by its dual.

\item We redefine $P_{\bullet \bullet} \in ss\Set$ analogously to how we redefined $N_\bullet \in s\Set$ (i.e.\! requiring the chosen objects of $(\bW_\C)_{\dfibn x}$ to be cofibrant and requiring the chosen objects of $(\bW_\D)_{y \dcofibn}$ to be fibrant).

\item Let us clarify why the asserted maps from $\diag(P_{\bullet \bullet})$ are weak equivalences in $s\Set_\KQ$.\footnote{We work in the modified situation of $(\C^c+\D^f)$, but the clarifications equally well clarify the argument given in the original proof of \cite[Proposition 7.3]{Mandell}; these clarifications actually come from private correspondence with Mandell regarding the original proof.}
\begin{itemize}

\item To see that the map $\diag(P_{\bullet \bullet}) \ra \Nerve \left( \Funp(\tilde{\word{3}},(\C^c+\D^f))^\bW \right)$ is a weak equivalence in $s\Set_\KQ$, let us define the object $\const \left(  \Nerve \left( \Funp(\tilde{\word{3}},(\C^c+\D^f))^\bW \right)  \right) \in ss\Set$ by precomposition with the projection $\bD^{op} \times \bD^{op} \ra \bD^{op}$ to the second factor.  This admits an evident map
\[ P_{\bullet \bullet} \ra \const \left(  \Nerve \left( \Funp(\tilde{\word{3}},(\C^c+\D^f))^\bW \right)  \right) \]
which yields the original map when we apply $\diag : ss\Set \ra s\Set$.  By \cite[Chapter IV, Proposition 1.9]{GJ} it suffices to show that this is a levelwise weak equivalence in $s\Set_\KQ$.  In level $n$, this is given by the map
\[ P_{\bullet ,n} \ra \Nerve \left( \Funp(\tilde{\word{3}},(\C^c+\D^f))^\bW \right)_n , \]
whose target is a discrete (i.e.\! constant) simplicial set.  Moreover, the fiber over any point of the target is the nerve of a category with an initial object, and hence it is weakly contractible in $s\Set_\KQ$.

\item To see that the map $\diag(P_{\bullet \bullet}) \ra N_\bullet$ is a weak equivalence in $s\Set_\KQ$, let us define $N_{\bullet \bullet} \in ss\Set$ by
\[ N_{m,n} = \coprod_{(\alpha,\beta) \in \hom_\Cat \left( [n],(\bW^c_\C)_{\dfibn x} \right) \times \hom_\Cat \left( [m],(\bW^f_\D)_{y \dcofibn} \right) } \hom_{(\C^c+\D^f)}(\alpha(n),\beta(0)) . \]
Note that this has $\diag(N_{\bullet \bullet}) \cong N_\bullet$, and moreover it admits an evident map $P_{\bullet \bullet} \ra N_{\bullet \bullet}$ which yields the original map when we apply $\diag : ss\Set \ra s\Set$.  Hence, again by \cite[Chapter IV, Proposition 1.9]{GJ}, it suffices to show that for each $n \geq 0$, the map $P_{n,\bullet} \ra N_{n,\bullet}$ is a weak equivalence in $s\Set_\KQ$.  In fact, it is not hard to see that this last map admits a section which defines a homotopy equivalence in $s\Set_\KQ$.

\end{itemize}

\end{itemize}

\item The third pair of vertical arrows are weak equivalences by \cref{special-3 to 3}.

\item The fourth pair of vertical arrows are weak equivalences by Propositions \ref{homotopical three-arrow calculus helps compute hammock sset} \and \ref{the two cases have homotopical three-arrow calculi}.

\end{itemize}
Thus, the map $\ham(\C^c+\D^f) \ra \cocart(\ham(F^c))$ is indeed a weak equivalence in $((\Cat_{s\Set})_{\Bergner})_{/[1]}$.
\end{proof}

\appendix

\section{A history of partial answers to \cref{main question}}\label{section history}

In this appendix, we survey the results that are either explicitly stated in the existing literature or can be extracted therefrom surrounding the question of providing external, homotopy-theoretic meaning to the notions of Quillen adjunctions and Quillen equivalences.

\subsection{Derived adjunctions and derived equivalences}

In \cite{QuillenHA}, Quillen proved the following results (which appear together as \cite[Chapter I, \sec 4, Theorem 3]{QuillenHA}).
\begin{itemize}

\item A Quillen adjunction induces a canonical adjunction between homotopy categories, called the \textit{derived adjunction} of the Quillen adjunction.

\item In the special case of a Quillen equivalence, the derived adjunction actually defines an equivalence of categories, called the \textit{derived equivalence} of the Quillen equivalence.

\end{itemize}

\subsection{Enhancements to $\ho(s\Set_\KQ)$-enriched categories}

In \cite{DKCalc}, Dwyer--Kan introduced their \textit{hammock localization} construction, which takes any relative category -- and hence in particular a model category -- and yields a category enriched in simplicial sets.  As $s\Set$-enriched categories provide a model for ``the homotopy theory of homotopy theories'', this laid the foundations for the following enhancements of Quillen's results that they proved (which appear together by combining \cite[Propositions 5.4 and 4.4]{DKFunc}).\footnote{In the statement of \cite[Proposition 5.4]{DKFunc}, we should be using a \textit{cosimplicial} resolution of the source and a \textit{simplicial} resolution of the target.}\footnote{The proof of \cite[Proposition 4.4]{DKFunc} contains a mistake, which is both explained and corrected in \cite{DuggerClass} and is corrected in \cite[\sec 7]{Mandell}.}
\begin{itemize}

\item A Quillen adjunction $F: \C \adjarr \D : G$ induces weak equivalences
\[ \hom_{\ham(\C)}(x,G(y)) \approx \hom_{\ham(\D)}(F(x),y) \]
in $s\Set_\KQ$ for every $x \in \C$ and every $y \in \D$.

\item A Quillen equivalence $F : \C \adjarr \D : G$ induces weak equivalences $\ham(F^c) : \ham(\C^c) \we \ham(\D^c)$ and $\ham(\C^f) \lwe \ham(\D^f) : \ham(G^f)$ in $(\Cat_{s\Set})_\Bergner$.  As illustrated in \cref{figure hexagon}, it follows that $\ham(\C)$ and $\ham(\D)$ are therefore weakly equivalent objects of $(\Cat_{s\Set})_\Bergner$.

\end{itemize}

Note that the first result does \textit{not} posit the existence of any sort of adjunction.  Indeed, this is a very subtle issue.  What we have so far is the diagram
\begin{figure}[h]
\[ \begin{tikzcd}
& \ham(\C^c) \arrow{r}{\ham(F^c)} & \ham(\D^c) \arrow[hook]{rd}[sloped, pos=0.4]{\approx} \\
\ham(\C) \arrow[hookleftarrow]{ru}[sloped, pos=0.6]{\approx} \arrow[hookleftarrow]{rd}[sloped, swap, pos=0.6]{\approx} & & & \ham(\D), \\
& \ham(\C^f) & \ham(\D^f) \arrow{l}{\ham(G^f)} \arrow[hook]{ru}[sloped, swap, pos=0.4]{\approx}
\end{tikzcd} \]
\caption{The hexagon in $(\Cat_{s\Set})_\Bergner$ arising from a Quillen adjunction.}\label{figure hexagon}
\end{figure}
in $(\Cat_{s\Set})_\Bergner$ of \cref{figure hexagon}, in which the fact that the indicated inclusions are weak equivalences follows from \cite[Proposition 5.2]{DKFunc} (or \cref{inclusion of co/fibrants induces BK weak equivalence}, see \cref{acknowledge DK}).

Now, a weak equivalence in $(\Cat_{s\Set})_\Bergner$ induces an equivalence of $\ho(s\Set_\KQ)$-enriched categories.\footnote{Equivalences of enriched categories are precisely the enriched functors which are essentially surjective on objects and induce isomorphisms on hom-objects (see \cite[\sec 1.11]{Kelly}).}  Hence, if we apply the ``enriched homotopy category'' functor $\hoenr : \Cat_{s\Set} \ra \Cat_{\ho(s\Set_\KQ)}$ to the above diagram, we can choose enriched inverse equivalences to the upper-left and lower-right inclusions, and then the upper and lower composites will respectively be candidates for the left and right adjoints of a $\ho(s\Set_\KQ)$-enriched adjunction between $\hoenr(\ham(\C))$ and $\hoenr(\ham(\D))$.

However, things are still not so clean as this.  The weak equivalences between corresponding hom-objects in $\ham(\C)$ and $\ham(\D)$ pass through the co/simplicial resolutions of \cite[4.3]{DKFunc}, and apparently nowhere in the literature are these shown to give \textit{functorially} weakly equivalent simplicial sets to the hom-objects in the hammock localizations, at least not in full generality.  In fact, the main purpose of \cite{LowFunc} is to show that these weak equivalences are indeed functorial (in $\ho(s\Set_\KQ)$) when the model category admits functorial factorizations (although Low mentions in that paper that he intends to return to the general case in future work).

But even if these weak equivalences were shown to be functorial, we still would not immediately obtain a $\ho(s\Set_\KQ)$-enriched adjunction.  Rather, we would need to choose our enriched inverse equivalences to be enriched \textit{adjoint} equivalences, in order to select preferred and functorial isomorphisms in $\hoenr(\ham(\C))$ and $\hoenr(\ham(\D))$ (via the unit or counit) between objects and their images under the retractions.\footnote{Adjoint equivalences would be guaranteed by the existence of functorial factorizations in $\C$ and $\D$ (or even just functorial cofibrant replacement in $\C$ and functorial fibrant replacement in $\D$), but such assumptions are unnecessary since we are ultimately only working at the $\ho(s\Set_\KQ)$-enriched level anyways: just as in ordinary category theory, an enriched functor is an enriched equivalence if and only if it admits an enriched adjoint equivalence (again see \cite[\sec 1.11]{Kelly}).}

\subsection{Enhancements to quasicategories}

It has already been established in the literature that certain Quillen adjunctions satisfying additional hypotheses induce adjunctions of quasicategories.

\subsubsection{Model categories with functorial replacements}

Lurie proves as \cite[Proposition 5.2.2.8]{LurieHTT} that a pair of functors between quasicategories are adjoints if and only if there exists a ``unit transformation'' with the expected behavior at the level of $\ho(s\Set_\KQ)$-enriched homotopy categories (see \cite[Definition 5.2.2.7]{LurieHTT}).

However, it is a subtle matter to obtain such a unit transformation.  Note that a Quillen adjunction $F : \C \adjarr \D : G$ gives rise to a unit transformation $\id_\C \ra GF$ of endofunctors on the underlying category $\C$, but its target $GF$ will not generally be a relative endofunctor.  The standard fix is to take cofibrant replacements in $\C$ before applying $F$ and fibrant replacements in $\D$ before applying $G$.  Of course, in order to obtain a unit transformation, these replacements must be functorial.  Let us assume we are in the usual situation in which such replacement functors exist, namely that they are obtained as special cases of functorial factorizations; we denote them by $\bbQ^\C : \C \ra \C^c \hookra \C$ and $\bbR^\D : \D \ra \D^f \hookra \D$, and we denote their corresponding replacement transformations by $\bbQ^\C \xra{q^\C} \id_\C$ and $\id_\D \xra{r^\D} \bbR^\D$.

Now, we are interested in obtaining a unit map for the relative endofunctor $G \bbR^\D F \bbQ^\C$ on $(\C,\bW_\C)$, at least at the level of its underlying quasicategory.  The first thing to note here is that we cannot proceed by passing through hammock localizations, since the functor $\ham : \RelCat \ra \Cat_{s\Set}$ does not preserve natural transformations.\footnote{Rather, given $\C_1,\C_2 \in \RelCat$ and a morphism $F_1 \ra F_2$ in $\Fun(\C_1,\C_2)^\bW$, for any $x, y\in \C_1$ we obtain a natural cospan of weak equivalences $\hom_{\ham(\C_2)}(F_1(x),F_1(y)) \we \hom_{\ham(\C_2)}(F_1(x),F_2(y)) \lwe \hom_{\hom(\C_2)}(F_2(x),F_2(y))$ in $s\Set_\KQ$, and combining this with the span $\hom_{\ham(\C_2)}(F_1(x),F_1(y)) \la \hom_{\ham(\C_1)}(x,y) \ra \hom_{\ham(\C_2)}(F_2(x),F_2(y))$ yields a square which commutes up to a specified homotopy (see \cite[Propositions 3.5 and 3.3]{DKCalc}).}  On the other hand, the relative functor $\NerveRezk : \RelCat \ra ss\Set_\Rezk$ preserves products (being pointwise corepresented), and from this it is not hard to see that it preserves natural transformations and takes natural weak equivalences to natural equivalences (in the evident internal sense in $ss\Set_\Rezk$); since the model category $ss\Set_\Rezk$ is compatibly cartesian closed (see \cite[Theorem 7.2]{RezkCSS}) and all its objects are cofibrant, we view this as an acceptable substitute.  Hence, up to the contractible ambiguity in the various functors between models for ``the homotopy theory of homotopy theories'', we may consider a natural transformation or natural weak equivalence between relative functors between relative categories as giving natural transformations and natural equivalences between the corresponding functors between their underlying quasicategories.

From here, the most direct way to proceed would be to obtain the unit map from the natural zigzag
\[ x \xla[\approx]{q^\C_x} \bbQ^\C(x) \xra{\eta_{\bbQ^\C(x)}} G \left( F \left( \bbQ^\C(x) \right) \right) \xra{G\left( r^\D_{F\left( \bbQ^\C(x) \right)} \right)} G \left( \bbR^\D \left( F \left( \bbQ^\C(x) \right) \right) \right) \]
in $\C$ in which the backwards arrow is a weak equivalence; assembling these across all $x \in \C$, we obtain a span
\[ \id_\C \xla[\approx]{q^\C} \bbQ^\C \xra{G \left( r^\D \right) \circ \eta} G \bbR^\D F \bbQ^\C \]
between relative endofunctors on $(\C,\bW_\C)$ in which the backwards arrow is a natural weak equivalence.  Passing through $ss\Set_\Rezk$ as discussed above (and implicitly identifying the two different ways of passing from $\RelCat_\BarKan$ to $s\Set_\Joyal$), we obtain a span
\[ \id_{\uq(\C)} \xla[\sim]{\uq \left( q^\C \right)} \uq \left( \bbQ^\C \right) \xra{\uq \left( G \left( r^\D \right) \circ \eta \right)} \uq \left( G \bbR^\D F \bbQ^\C \right) \]
in which the backwards arrow is an equivalence.\footnote{Note that we are now working \textit{internally} to a quasicategory, namely the quasicategory of endofunctors of $\uq(\C)$.}  Hence, we can obtain a candidate unit transformation $\id_{\uq(\C)} \ra \uq \left( G \bbR^\D F \bbQ^\C \right)$, which one might then hope to verify satisfies the hypotheses of \cite[Definition 5.2.2.7]{LurieHTT} using e.g.\! the co/simplicial resolutions of \cite{DKFunc}.  Of course, this requires knowing that the hom-objects obtained from co/simplicial resolutions are indeed \textit{functorially} weakly equivalent to the hom-objects in the hammock localizations, but at least this follows from \cite{LowFunc} in the case that $\C$ and $\D$ both admit functorial factorizations, as mentioned above.

\begin{rem}
This approach would also work if the cofibrant replacement functor $\bbQ^\C : \C \ra \C$ were augmented (instead of coaugmented), and in fact we would also obtain a candidate unit transformation if the fibrant replacement functor $\bbR^\D : \D \ra \D$ were coaugmented (instead of augmented).  On the other hand, because of the way model categories are set up, it seems that such replacement functors do not arise very frequently in practice.
\end{rem}

\begin{rem}
Of course, if all objects of $\C$ are cofibrant then the identity functor can serve as a cofibrant replacement functor; a dual observation holds for $\D$.
\end{rem}

\begin{rem}
Actually, slightly more cleverly, we can use a similar argument to the one given above to obtain a natural transformation between the standard inclusion $\C^c \hookra \C$ and the composite $\C^c \xra{F^c} \D^c \hookra \D \xra{\bbR^\D} \D^f \xra{G} \C$; this yields a natural transformation between functors from $\uq(\C^c)$ to $\uq(\C)$, and (horizontally) precomposing with an inverse to the equivalence $\uq(\C^c) \xra{\sim} \uq(\C)$ yields a candidate unit transformation, all without requiring that $\C$ admit any sort of cofibrant replacement functor.  Dually, one can obtain a candidate counit transformation if one assumes that $\C$ has a cofibrant replacement functor but without assuming that $\D$ admit any sort of fibrant replacement functor.
\end{rem}

\begin{rem}
Instead of assuming the existence of any appropriate co/fibrant replacement functors, one might alternatively extract inverse equivalences $\uq(\C) \xra{\sim} \uq(\C^c)$ and $\uq(\D) \xra{\sim} \uq(\D^f)$ at the level of underlying quasicategories.  However, it appears that the original adjunction $F \adj G$ will be entirely lost by this point, and hence that one cannot hope to provide the desired unit transformation in full generality using this approach.
\end{rem}

\subsubsection{Simplicial model categories}

Dwyer--Kan prove that given a simplicial model category $\C_\bullet$, the two possible notions of ``underlying homotopy theory'' agree: the full $s\Set$-enriched subcategory $\C^{cf}_\bullet$ of bifibrant objects is equivalent (via a zigzag of weak equivalences in $(\Cat_{s\Set})_\Bergner$) to the hammock localization $\ham(\C)$ of the underlying model category (see \cite[Proposition 4.8]{DKFunc}).\footnote{Note that in the statement of \cite[Proposition 4.8]{DKFunc}, the right arrow should also be labeled as a weak equivalence in $(\Cat_{s\Set})_\Bergner$, there called simply a ``weak equivalence''.}  This paved the way for the following enhancement of their results.

First of all, Lurie proves as \cite[Proposition 5.2.4.6]{LurieHTT} -- and Riehl--Verity later re-prove as \cite[Theorem 6.2.1]{RV} -- that a simplicial Quillen adjunction of simplicial model categories $F_\bullet : \C_\bullet \adjarr \D_\bullet : G_\bullet$ (that is, an enriched adjunction in $\Cat_{s\Set}$ which is moreover a Quillen adjunction on underlying model categories) induces an adjunction between the quasicategories $\Nervehc(\C_\bullet^{cf})$ and $\Nervehc(\D_\bullet^{cf})$.\footnote{Neither of these sources defines model categories to come equipped with functorial factorizations, although Lurie assumes bicompleteness (whereas Riehl--Verity only assume finite bicompleteness).  On the other hand, Lurie's proof readily adapts to the more general case.}  (Note that the objects $\C_\bullet^{cf}, \D_\bullet^{cf} \in (\Cat_{s\Set})_\Bergner$ are already fibrant, and hence do not require fibrant replacement.)

Moreover, there are various results concerning replacing model categories and Quillen equivalences by simplicial ones.

\begin{itemize}

\item In \cite{DuggerSimp}, Dugger shows that a model category which is left proper and is additionally either cellular or combinatorial admits a left Quillen equivalence to a simplicial model category (see \cite[Theorem 1.2 or 6.1]{DuggerSimp}).

\item In \cite{RSS}, Rezk--Schwede--Shipley work with model categories that are left proper, cofibrantly generated (under a slightly stronger definition than the usual one, see \cite[Definition 8.1]{RSS}), and satisfy their ``realization axiom'' (see \cite[Axiom 3.4]{RSS}), and prove

\begin{itemize}

\item that every such model category admits a left Quillen equivalence to a simplicial model category (see \cite[Theorem 3.6]{RSS}), and

\item that a Quillen adjunction between such model categories induces a simplicial Quillen adjunction between their replacements by simplicial model categories (see \cite[Proposition 6.1]{RSS}).

\end{itemize}

\end{itemize}
Whenever these results can be used to upgrade a Quillen adjunction to a simplicial Quillen adjunction (see \cite[\sec A]{BlumbergRiehl} for an expanded summary of these techniques), then by combining Lurie's result with the Dwyer--Kan result cited earlier (that Quillen equivalences induce weak equivalences in $(\Cat_{s\Set})_\Bergner$), we obtain from the original Quillen adjunction an adjunction of underlying quasicategories.

\bibliographystyle{amsalpha}
\bibliography{adjns}{}

\providecommand{\bysame}{\leavevmode\hbox to3em{\hrulefill}\thinspace}
\providecommand{\MR}{\relax\ifhmode\unskip\space\fi MR }
\providecommand{\MRhref}[2]{%
  \href{http://www.ams.org/mathscinet-getitem?mr=#1}{#2}
}
\providecommand{\href}[2]{#2}
\begin{thebibliography}{DK80b}

\bibitem[Ber07]{Bergner}
Julia~E. Bergner, \emph{A model category structure on the category of
  simplicial categories}, Trans. Amer. Math. Soc. \textbf{359} (2007), no.~5,
  2043--2058.

\bibitem[BHH]{BHH}
Ilan Barnea, Yonatan Harpaz, and Geoffroy Horel, \emph{{On the pro-category of
  a weak fibration category}}, in preparation.

\bibitem[BK12a]{BK-simploc}
C.~Barwick and D.~M. Kan, \emph{A characterization of simplicial localization
  functors and a discussion of {DK} equivalences}, Indag. Math. (N.S.)
  \textbf{23} (2012), no.~1-2, 69--79.

\bibitem[BK12b]{BK-relcats}
\bysame, \emph{Relative categories: another model for the homotopy theory of
  homotopy theories}, Indag. Math. (N.S.) \textbf{23} (2012), no.~1-2, 42--68.

\bibitem[BR14]{BlumbergRiehl}
Andrew~J. Blumberg and Emily Riehl, \emph{Homotopical resolutions associated to
  deformable adjunctions}, Algebr. Geom. Topol. \textbf{14} (2014), no.~5,
  3021--3048.

\bibitem[BSP]{BSP}
Clark Barwick and Christopher Schommer-Pries, \emph{{On the unicity of the
  homotopy theory of higher categories}}, available at {\tt arXiv:1112.0040},
  v4.

\bibitem[Cor82]{Cordier}
Jean-Marc Cordier, \emph{Sur la notion de diagramme homotopiquement
  coh\'erent}, Cahiers Topologie G\'eom. Diff\'erentielle \textbf{23} (1982),
  no.~1, 93--112, Third Colloquium on Categories, Part VI (Amiens, 1980).

\bibitem[DK80a]{DKCalc}
W.~G. Dwyer and D.~M. Kan, \emph{Calculating simplicial localizations}, J. Pure
  Appl. Algebra \textbf{18} (1980), no.~1, 17--35.

\bibitem[DK80b]{DKFunc}
\bysame, \emph{Function complexes in homotopical algebra}, Topology \textbf{19}
  (1980), no.~4, 427--440.

\bibitem[Dug]{DuggerClass}
Daniel Dugger, \emph{{Classification spaces of maps in model categories}},
  available at {\tt arXiv:0604537}, v2.

\bibitem[Dug01]{DuggerSimp}
\bysame, \emph{Replacing model categories with simplicial ones}, Trans. Amer.
  Math. Soc. \textbf{353} (2001), no.~12, 5003--5027 (electronic).

\bibitem[Ehr68]{EhresmannSketch}
Charles Ehresmann, \emph{Esquisses et types des structures alg\'ebriques}, Bul.
  Inst. Politehn. Ia\c si (N.S.) \textbf{14 (18)} (1968), no.~fasc. 1-2, 1--14.

\bibitem[GJ99]{GJ}
Paul~G. Goerss and John~F. Jardine, \emph{Simplicial homotopy theory}, Progress
  in Mathematics, vol. 174, Birkh\"auser Verlag, Basel, 1999.

\bibitem[Hin05]{Hinich}
Vladimir Hinich, \emph{Deformations of sheaves of algebras}, Adv. Math.
  \textbf{195} (2005), no.~1, 102--164.

\bibitem[Hir03]{Hirsch}
Philip~S. Hirschhorn, \emph{Model categories and their localizations},
  Mathematical Surveys and Monographs, vol.~99, American Mathematical Society,
  Providence, RI, 2003.

\bibitem[Kel05]{Kelly}
G.~M. Kelly, \emph{Basic concepts of enriched category theory}, Repr. Theory
  Appl. Categ. (2005), no.~10, vi+137, Reprint of the 1982 original [Cambridge
  Univ. Press, Cambridge; MR0651714].

\bibitem[LMG]{LowMG}
Zhen~Lin Low and Aaron Mazel-Gee, \emph{{From fractions to complete Segal
  spaces}}, available at {\tt arxiv:1409.8192}, v2.

\bibitem[Low]{LowFunc}
Zhen~Lin Low, \emph{{Revisiting function complexes and simplicial
  localisation}}, available at {\tt arXiv:1409.8062}, v1.

\bibitem[Lur09]{LurieHTT}
Jacob Lurie, \emph{Higher topos theory}, Annals of Mathematics Studies, vol.
  170, Princeton University Press, Princeton, NJ, 2009, also available at {\tt
  http://math.harvard.edu/{$\thicksim$lurie}} and at {\tt arXiv:math/0608040},
  v4.

\bibitem[Man99]{Mandell}
Michael~A. Mandell, \emph{Equivalence of simplicial localizations of closed
  model categories}, J. Pure Appl. Algebra \textbf{142} (1999), no.~2,
  131--152.

\bibitem[ML98]{MacLaneCWM}
Saunders Mac~Lane, \emph{Categories for the working mathematician}, second ed.,
  Graduate Texts in Mathematics, vol.~5, Springer-Verlag, New York, 1998.

\bibitem[Qui67]{QuillenHA}
Daniel~G. Quillen, \emph{Homotopical algebra}, Lecture Notes in Mathematics,
  No. 43, Springer-Verlag, Berlin-New York, 1967.

\bibitem[Qui73]{QuillenAKTI}
Daniel Quillen, \emph{Higher algebraic {$K$}-theory. {I}}, Algebraic
  {$K$}-theory, {I}: {H}igher {$K$}-theories ({P}roc. {C}onf., {B}attelle
  {M}emorial {I}nst., {S}eattle, {W}ash., 1972), Springer, Berlin, 1973,
  pp.~85--147. Lecture Notes in Math., Vol. 341.

\bibitem[Rez01]{RezkCSS}
Charles Rezk, \emph{A model for the homotopy theory of homotopy theory}, Trans.
  Amer. Math. Soc. \textbf{353} (2001), no.~3, 973--1007 (electronic).

\bibitem[RSS01]{RSS}
Charles Rezk, Stefan Schwede, and Brooke Shipley, \emph{Simplicial structures
  on model categories and functors}, Amer. J. Math. \textbf{123} (2001), no.~3,
  551--575.

\bibitem[RV]{RV}
Emily Riehl and Dominic Verity, \emph{{The 2-category theory of
  quasi-categories}}, available at {\tt arXiv:1306.5144}, v2.

\end{thebibliography}

\end{document}